\documentclass[10pt,oneside,openright,a4paper,reqno]{amsart}
\usepackage{times}
\usepackage{hyperref}
\usepackage{amsmath}
\usepackage{amsfonts, mathrsfs}
\usepackage{amssymb}
\usepackage{paralist}
\usepackage{xcolor}
\usepackage{soul}

\newtheorem{lemma}{Lemma}
\newtheorem{theorem}{Theorem}
\theoremstyle{definition}
\newtheorem{assumption}{A -}

\newtheorem{definition}{Definition}
\newtheorem{remark}{Remark}
\newtheorem{corollary}{Corollary}

\DeclareMathOperator*{\dist}{dist}
\DeclareMathOperator*{\supp}{supp}
\title{$L^p$-estimates and regularity for SPDEs with monotone semilinearity
}
\author{Neelima}
\address{ School of Mathematics, University of Edinburgh, United Kingdom
and Ramjas College, University of Delhi, Delhi, India }
\email{N.Neelima@sms.ed.ac.uk}
\author{David \v{S}i\v{s}ka}
\address{School of Mathematics, University of Edinburgh, United Kingdom}
\email{D.Siska@ed.ac.uk}
\date{11th September 2019}


\thanks{This an electronic reprint of the article published in {\em Stoch PDE: Anal Comp} (2019). }
\thanks{
It is available
at \url{https://doi.org/10.1007/s40072-019-00150-w}.
It may differ from the published version in typographical detail.
}
\begin{document}

\begin{abstract}
Semilinear stochastic partial differential equations on 
bounded domains $\mathscr{D}$ are considered.
The semilinear term may have arbitrary polynomial growth as long as it is 
continuous and monotone except perhaps near the origin.
Typical examples are the stochastic Allen--Cahn and Ginzburg--Landau equations.
The first main result of this article are $L^p$-estimates for such equations.
The $L^p$-estimates are subsequently employed in obtaining higher regularity. 
This is motivated by ongoing work to obtain rate of 
convergence estimates for numerical approximations 
to such equations.
It is shown, under appropriate assumptions, 
that the solution is continuous in time with values 
in the Sobolev space $H^2(\mathscr{D}')$ and $\ell^2$-integrable with
values in $H^3(\mathscr{D}')$, 
for any compact $\mathscr{D}' \subset \mathscr{D}$.
Using results from $L^p$-theory of SPDEs obtained by Kim~\cite{kim04}
we get analogous results in weighted Sobolev spaces 
on the whole $\mathscr{D}$.
Finally it is shown that the solution is H\"older continuous in time
of order $\frac{1}{2} - \frac{2}{q}$ as a process with values
in a weighted $L^q$-space, where $q$ arises from the integrability 
assumptions imposed on the initial condition and forcing terms.
\end{abstract}

\maketitle

\tableofcontents

\section{Introduction}

The aim of this article is to obtain $L^p$-estimates and regularity of solutions
to the semilinear stochastic partial differential equation (SPDE)
\begin{equation}
\label{eq:spde}
\begin{split}
du_t & =(L_tu_t+f_t(u_t,\nabla u_t)+f_t^0)dt+\sum_{k\in \mathbb{N}}(M_t^ku_t+g_t^k)dW_t^k
\,\,\,\ \text{on}\,\,\, [0,T]\times \mathscr{D}\\
 u_t & =0 \,\,\, \text{on}\,\,\, \partial \mathscr{D}, \quad u_0 =\phi \,\,\, \text{on}\,\,\, \mathscr{D}.
 \end{split}
\end{equation}
where,
\begin{equation}\label{eq:def_L,M}
L_t u:= \sum_{j=1}^d\partial_j\Big(\sum_{i=1}^da_t^{ij}\partial_iu\Big)+\sum_{i=1}^db_t^i\partial_iu+c_tu \,\,\,\,\,\,\text{and}\,\,\,\, \,\,
M_t^k u :=\sum_{i=1}^d\sigma_t^{ik}\partial_iu+\mu_t^ku.
\end{equation}
Here $\mathscr{D}$ is a bounded domain in $\mathbb{R}^d$ and $W^k$ are independent Wiener
processes. 
The coefficients $a$ and $\sigma$ are assumed to satisfy 
stochastic parabolicity condition (and thus our equation is non-degenerate). 
Moreover all the coefficients $a$, $b, c$, $\sigma$ and $\mu$ are assumed to 
be measurable and bounded, 
$f = f_t(\omega, x, r, z)$ is measurable, continuous in $(r,z)$, monotone in 
$r$ except perhaps around the origin, Lipschitz continuous in $z$, 
bounded in $x$ and of polynomial growth in $r$ (of arbitrary order).
The forcing terms $f^0$ and $g$ are assumed to satisfy appropriate
integrability conditions. 
A typical example of equation fitting this setting is the stochastic Ginzburg--Landau equation.
In this case 
\[f(r) = -|r|^{\alpha-2}r\,, \,\,\,\alpha \geq 2.\]
To obtain higher interior regularity we will have to impose further regularity 
assumptions on the coefficients. 
To obtain regularity up to the boundary (in weighted Sobolev spaces) we will 
also need to impose regularity assumptions on the domain.
The assumptions will be formulated precisely in further sections.

The main aim of this article is to obtain regularity results for the solutions 
to the SPDE~\eqref{eq:spde}.
This is motivated by ongoing work to obtain rate of 
convergence estimates for numerical approximations 
to such equations.
For a semilinear equation it is natural to consider the term 
$f:=f(u,\nabla u) + f^0$ as a free term in an appropriate linear SPDE and
to use established methods and theory to obtain regularity for this linear SPDE.
Due to uniqueness of solutions to~\eqref{eq:spde}, see Lemma~\ref{lem:uniqueness}, 
we then get the same regularity for the semilinear equation~\eqref{eq:spde}.
However, for the theory of regularity of linear SPDEs to apply, we need to show that 
the new free term $f$ satisfies appropriate integrability conditions.
This would typically mean at least $L^2$-integrability. 
Since the semilinear term in~\eqref{eq:spde} is allowed to have arbitrary polynomial 
growth, it is clear that we need to obtain $L^p$-estimates for solution
to~\eqref{eq:spde} with $p\geq 2$ sufficiently large.
Note that if one attempts to do this using Sobolev embedding theorem then
one immediately runs into restrictions on the combination of dimension of $\mathscr{D}$ 
and the growth of the semilinear term. 

The main novelty of this article is in allowing arbitrary dimension of $\mathscr{D}$ and growth of the semilinear term, 
see Theorem~\ref{thm:sol_bound}. 
This is achieved by
using the monotonicity property of the semilinear term and 
a cutting argument to obtain the required $L^p$-estimate. 
Once these have been established we then obtain new spatial regularity 
results for the SPDE~\eqref{eq:spde},
these are both interior regularity and up-to-the-boundary regularity 
in weighed Sobolev spaces, 
see Theorems~\ref{thm:higher_regularity}
and~\ref{thm:Lp_regularity}.
Finally we have a new time regularity result (in weighted space again),
see Theorem~\ref{thm:time_reg}. 
These effectively say that under appropriate assumptions the SPDE~\eqref{eq:spde}
has two additional derivatives. 
It seems however that our method does not allow one to obtain arbitrarily high 
regularity (even for equation with smooth data and coefficients), 
see Remark~\ref{rem:regularity} for explanation.
Nevertheless, raising the regularity twice is enough to find the rate of convergence of various numerical approximations using the techniques from e.g. Gy\"ongy and Millet~\cite{millet09}.

Regularity of solutions to linear PDEs has been studied intensively, see e.g. Evans \cite{evans97}, Gilbarg and Trudinger \cite{gil83} for elliptic PDEs , Lady\v{z}enskaja et al. \cite{lady68} for parabolic PDEs and references therein. Regularity results for linear elliptic and parabolic PDEs in H\"older spaces can be found in Krylov \cite{krylov96}. Regularity of solutions to SPDEs has been an area of active interest for 
quite some time and here we point out some of the main results.
Regularity of solutions to linear SPDEs on the whole space has been proved in 
Rozovskii~\cite{Rozovskii90}.
On domains with a boundary the situation is much more involved and one cannot 
expect the same regularity up to the boundary as in the interior of the domain, 
see e.g. Examples 1.1 and 1.2 in Krylov~\cite{krylov94}.
After this observation two approaches to dealing with boundaries emerge:
one is to quantify the loss of regularity near the boundary 
using weighted Sobolev spaces. 
These allow oscillations and explosion of the spatial derivatives
of the solution near the boundary. 
The other approach is to side-step the problems created by the boundary
by restricting the class of
equations under consideration by imposing additional restriction on the 
noise term near the boundary 
(effectively disallowing stochastic forcing near the boundary), 
see Flandoli~\cite{flandoli90}.
Weighted Sobolev spaces have also been employed, in the context 
of $L^p$-thoery for linear SPDEs, by Kim~\cite{kim04}.
Unsurprisingly, there are fewer results for nonlinear SPDEs. 
Kim and Kim use the $L^p$-theory in~\cite{kim14} and~\cite{kim17}
to obtain regularity for quasilinear SPDEs where the coefficients are uniformly bounded.
Current results in Gerencs\'er~\cite{gerencser17} show
that for a class of SPDEs, including~\eqref{eq:spde}, there exists 
some H\"older exponent such that the solution is H\"older continuous in space
up to the boundary with this exponent. 
For interior regularity of a class of quasilinear equations 
associated with the ``$p$-Laplace'' operator see Breit~\cite{breit16}.
For SPDEs with drift given by the subgradient of a quasi-convex function 
and with sufficiently regular noise Gess~\cite{gess12} proves
higher regularity and existence of (analytically) strong solutions.
All the aforementioned work on regularity of nonlinear SPDEs has been done using the variational approach.  
For results obtained in the semigroup framework we refer the reader to the work of Jentzen and R\"ockner~\cite{jentzenrockner12} and
references therein. 
Regularity results for quasilinear PDEs of parabolic type can be found in \cite{lady68}. However, the results are obtained under the restrictions on the combination of dimension of $\mathscr D$ and the growth of the nonlinear term. 
Thus, to the best of our knowledge, our results are new even for deterministic semilinear PDEs with monotone semilinear term.   



The article is organised as follows: 
Section~\ref{sec: Lp} is devoted to the proof of Theorem~\ref{thm:sol_bound} which gives us the desired $L^p$-estimates for the solution to semilinear 
SPDE~\eqref{eq:spde}.
In Section~\ref{sec:regularity}, we first prove interior regularity for
the associated linear SPDE, see Theorem~\ref{thm:H_m_regularity}.   
We then use the results on interior regularity of the linear SPDE 
to prove Theorem~\ref{thm:higher_regularity}. 
In Section~\ref{sec:Lp_reg}, we prove regularity results up to the boundary and
time regularity in weighted Sobolev spaces using $L^p$-theory from Kim~\cite{kim04}. 
The main results and required assumptions are stated at the beginning of 
each section.

\section{$L^p$-estimates for the semilinear equation}
\label{sec: Lp}
Let $T>0$ be given, 
$(\Omega,\mathscr{F},(\mathscr{F}_t)_{t\in [0,T]},\mathbb{P})$ be a stochastic basis, 
$\mathscr{P}$ be the predictable $\sigma$-algebra and 
$W:=(W_t)_{t\in [0,T]}$ be an infinite dimensional Wiener martingale 
with respect to $(\mathscr{F}_t)_{t\in [0,T]}$, 
i.e. the coordinate processes $(W_t^k)_{t\in [0,T]},\,\, k\in \mathbb{N}$ 
are independent $\mathscr{F}_t$-adapted Wiener processes 
such that $W_t^k-W_s^k$ is independent of $\mathscr{F}_s$ for $s\leq t$. 
Further, let $\mathscr{D}$ be a bounded domain in $\mathbb{R}^d$ with 
Lipschitz boundary. We use standard notation for Lebesgue--Bochner 
and Sobolev spaces.
In general, if $X$ is a normed linear space then we will use $|\cdot|_X$ 
to denote the norm in this space. 
There are exceptions: if $x\in \mathbb{R}^d$ then $|x|$ denotes the Euclidean norm. 
For Lebesgue and Sobolev spaces over the entire domain $\mathscr{D}$ we 
will omit the dependence on $\mathscr{D}$. 
So e.g. if $h\in L^p(\mathscr{D})$ then we will write $|h|_{L^p}$ 
for $|h|_{L^p(\mathscr{D})}$. 
If $h\in L^p((0,T); L^p(\mathscr{D}))$ then we use 
$\|h\|_{L^p}$ to denote the norm.
Throughout this article \textcolor{red}{$C$} denotes a generic constant 
that may change from line to line.

Let $n\in \{0\}\cup \mathbb{N}$ and fix constants 
$K>0$, $\kappa > 0$, $\alpha \geq 2$ and $p \geq \alpha$. 
We assume the following: 

\begin{assumption} \label{ass:boundedness}
For any $i,j=1,\ldots,d$, the coefficients $a^{ij}, b^i$ and $c$ 
are real-valued, 
$\mathscr{P} \times \mathscr{B}(\mathscr{D})$-measurable and are bounded by $K$. 
The coefficients $\sigma^i=(\sigma^{ik})_{k=1}^\infty$, $\mu=(\mu^k)_{k=1}^\infty$ 
are $\ell^2$-valued,
$\mathscr{P} \times \mathscr{B}(\mathscr{D})$-measurable and almost surely
\[
\sum_{i=1}^d\sum_{k\in \mathbb{N}} |\sigma_t^{ik}(x)|^2+\sum_{k\in\mathbb{N}}|\mu_t^k(x)|^2 \leq K \quad \forall t \in [0,T], x\in \mathscr{D}.
\] 
\end{assumption}


\begin{assumption} \label{ass:parabolicity}
Almost surely
\[
\sum_{i,j=1}^d \Big(a_t^{ij}(x)-\frac{1}{2}\sum_{k\in\mathbb{N}}\sigma_t^{ik}(x)\sigma_t^{jk}(x)\Big)\xi_i \xi_j \geq \kappa|\xi|^2 \quad 
\forall t \in [0,T], x\in \mathscr{D}, \xi \in \mathbb{R}^d\,.
\]
\end{assumption}

\begin{assumption} \label{ass:f}
The function 
$f = f_t(\omega, x, r, z)$ 
is $\mathscr{P} \times \mathscr{B}(\mathscr{D}) \times \mathscr{B}(\mathbb{R})\times \mathscr{B}(\mathbb{R}^d)$-measurable,
it is continuous in $(r,z)$ almost surely for all $t$ and $x$.
Furthermore, almost surely
\begin{equation*}
\begin{split}
(r-r')(f_t(x,r,z)-f_t(x,r',z)) & \leq K|r-r'|^2,  \\
|f_t(x,r,z)-f_t(x,r,z')| & \leq K|z-z'|, \\
|f_t(x,r,z)| & \leq K(1+|r|)^{\alpha-1}
\end{split}
\end{equation*}
 for all $t, x, r,r', z, z' $.
\end{assumption}

\begin{assumption}\label{ass:initial}
$\phi \in L^p(\Omega, \mathscr{F}_0; L^p(\mathscr{D}))$, $f^0 \in L^p(\Omega \times (0,T), \mathscr{P};L^p(\mathscr{D}))$ and $g \in L^p(\Omega \times (0,T), \mathscr{P};L^p(\mathscr{D};\ell^2))$.
\end{assumption}

\begin{remark} \label{rem:decrease}
Without loss of generality, we may assume that almost surely for all $t$, $x$ and $z$ the function $r\mapsto f_t(x,r,z)$ is decreasing. 
If not, then \eqref{eq:spde} can be rewritten by replacing $f_t(x,r,z)$ with $\bar{f}_t(x,r,z):=f_t(x,r,z)-Kr$ and $c_t(x)$ with $\bar{c}_t(x):=c_t(x)+K$, where using Assumption A - \ref{ass:f}, $\bar{f}$ is decreasing in $r$.

Further, we may assume that almost surely for all $t$ and $x$, $f_t(x,0,0)=0$. Otherwise, we can replace $f_t(x,r,z)$ in \eqref{eq:spde} by $\tilde{f}_t(x,r,z):=f_t(x,r,z)-f_t(x,0,0)$ and $f_t^0$ by $\tilde{f}_t^0(x):=f_t^0(x)+f_t(x,0,0)$.
\end{remark}

\begin{definition}[$L^2$-Solution]\label{def:solution}
An adapted, continuous $L^2(\mathscr{D})$-valued process is said to be a solution of stochastic partial differential equation \eqref{eq:spde} if
\begin{enumerate}[(i)]
\item $dt\times \mathbb{P}$ almost everywhere $u \in L^\alpha(\mathscr{D}) \cap H_0^1(\mathscr{D})$ and
$$\mathbb{E}\int_0^T(|u_t|_{L^\alpha}^\alpha+|u_t|_{H_0^1}^2)\, dt < \infty\,,$$
\item almost surely for every $t\in [0,T]$ and $\xi \in C_0^\infty(\mathscr{D})$, 
$$(u_t,\xi) = (u_0,\xi) + \int_0^t \langle L_s(u_s)+f_s(u_s,\nabla u_s)+f_s^0,\xi \rangle ds + \sum_{k\in\mathbb{N}} \int_0^t (\xi,M_s^k(u_s)+g_s^k)dW_s^k.$$ 
\end{enumerate}
\end{definition}

The following theorem is the main result of this section.
\begin{theorem} \label{thm:sol_bound}
If Assumptions A-\ref{ass:boundedness} to A-\ref{ass:initial} hold, then there exists a unique solution $u$ to \eqref{eq:spde} and
\begin{equation}  \label{eq:p_bound}
\begin{split}
\mathbb{E}\sup_{0\leq t \leq T}|u_t|_{L^p}^p &+\mathbb{E}\int_0^T\int_\mathscr{D}|\nabla u_s|^2|u_s|^{p-2}dxds \\
& \leq C\mathbb{E}\Big(|\phi|_{L^p}^p+\Vert f^0 \Vert _{L^p}^p+\Vert|g|_{\ell^2}\Vert^p_{L^p}\Big),
\end{split}
\end{equation}
where $C=C(d,p,K,\kappa,T)$.
\end{theorem}

The rest of Section~\ref{sec: Lp} is devoted to proving Theorem~\ref{thm:sol_bound} but we give a brief outline of the proof here.
\begin{enumerate}
\item We replace the semilinear term $f$ 
by truncations $f^m$, depending on some $m\in \mathbb{N}$,
chosen in such a way that that the monotonicity is preserved and $f^m$ are bounded. 
By standard theory of stochastic evolution equations we obtain $u^m$ which 
are solutions to the SPDE with $f$ replaced with $f^m$.
\item We now wish to get the estimate~\eqref{eq:p_bound} for these $u^m$ (uniformly in $m$). 
If we were allowed to apply It\^o's formula directly to $r \mapsto |r|^p$ and 
the process $u^m_t(x)$ and to integrate over $\mathscr{D}$ then~\eqref{eq:p_bound} for $u^m$ would follow from A-\ref{ass:boundedness}, A-\ref{ass:parabolicity} 
and A-\ref{ass:f}.   
\item Since, of course, this is not allowed we instead consider an appropriate bounded
smooth approximation $\phi_n$ to $r\mapsto |r|^p$ and use the It\^o formula from Krylov~\cite{krylov13}.
We then establish an estimate similar to~\eqref{eq:p_bound} but for $\phi_n(u^m)$
instead of $|u^m|^p$ 
and with the right-hand-side still depending on $m$ but independent of $n$.
See Lemma~\ref{lem:bound_with_m}. 
This allows us to take the limit $n\to \infty$ and to use the monotonicity
of $r\mapsto f^m_t(x,r,z)$ to obtain~\eqref{eq:p_bound} for $u^m$.
See Lemma~\ref{lem:cut bound}.
\item The final step is then to use compactness argument to obtain $u$ as a 
weak limit of $(u^m)_{m\in \mathbb{N}}$, see Lemma~\ref{lem:weak limits}, and the usual monotonicity argument
to show that $u$ satisfies~\eqref{eq:spde}.
Fatou's lemma will then yield~\eqref{eq:p_bound} for $u$.	
\end{enumerate}

Before proceeding with the proof of Theorem~\ref{thm:sol_bound}, we observe the following:

\begin{remark} \label{rem:coercivity}
Assumptions A-\ref{ass:boundedness} and A-\ref{ass:parabolicity} imply, after some computations using H\"older's and Young's inequalities, the existence of a constant $K'$ depending on $K,d$ and $\kappa$ only such that almost surely for all $t\in[0,T]$ and  $w,w' \in H_0^1(\mathscr{D})$, 

\begin{equation*}
2 \langle L_tw+f_t^0,w\rangle + \sum_{k\in\mathbb{N}}|M_t^k w + g_t^k|^2_{L^2}+\kappa |w|^2_{H_0^1} \leq K'\Big[|f_t^0|_{L^2}^2+\big||g_t|_{\ell^2}\big|^2_{L^2}+|w|^2_{L^2} \Big]
\end{equation*}
and
\begin{equation*}
 2 \langle L_tw-L_tw',w-w'\rangle + \sum_{k\in\mathbb{N}}|M_t^kw-M_t^kw'|^2_{L^2}+\kappa |w-w'|^2_{H_0^1}\leq K'|w-w'|^2_{L^2} \,\,.
\end{equation*}
\end{remark}

\begin{lemma}[Uniqueness]\label{lem:uniqueness}
The solution to~\eqref{eq:spde} is unique in the sense that if
$u$ and $\bar{u}$ both satisfy~\eqref{eq:spde} then
\[
\mathbb{P}\Big(\sup_{t\leq T}|u_t - \bar{u}_t|_{L^2}=0\Big) = 1.
\]	
\end{lemma}
\begin{proof}
Let $u$ and $\bar{u}$ be  two solutions of \eqref{eq:spde} in the sense of Definition \ref{def:solution}. Then,
\begin{equation}\label{eq:unique}
\begin{split}
u_t - \bar u_t =\int_0^t \left(L_s(u_s)-L_s(\bar{u}_s)+f_s(u_s, \nabla u_s)-f_s(\bar{u}_s, \nabla \bar{u}_s)\right)\,ds  \\
+\sum_{k\in\mathbb{N}} \int_0^t \left(M_s^k(u_s)-M_s^k(\bar{u}_s)\right)\,dW_s^k  
\end{split}
\end{equation}
almost surely for all $ t \in [0,T]$.
Using Remark \ref{rem:decrease}, Assumption A-\ref{ass:f} and Young's inequality, we get
\begin{equation}\label{eq:lipschitz}
\begin{split}
&\langle f_t(u_t, \nabla u_t)-f_t(\bar{u}_t, \nabla \bar{u}_t),u_t-\bar{u}_t\rangle \\
&=\langle f_t(u_t, \nabla u_t)-f_t(\bar{u}_t, \nabla u_t)+ f_t(\bar{u}_t, \nabla u_t)-f_t(\bar{u}_t, \nabla \bar{u}_t),u_t-\bar{u}_t\rangle \\
&\leq \frac{\kappa}{2} |\nabla(u_t-\bar{u}_t)|^2_{L^2}+N|u_t-\bar{u}_t|^2_{L^2} \, .
\end{split}
\end{equation}
Using the product rule and applying It\^o's formula for the the square of the norm to~\eqref{eq:unique}, see Gy\"ongy and \v{S}i\v{s}ka~\cite{gyongy17} or Pardoux~\cite[Chapitre 2, Theoreme 5.2]{pardoux75}, we obtain
\begin{equation} \label{eq:ito_product}
\begin{split}
d&\Big(e^{-K''t}|u_t-\bar{u}_t|_{L^2}^2 \Big)
= e^{-K''t}\big[d|u_t-\bar{u}_t|_{L^2}^2-K''|u_t-\bar{u}_t|_{L^2}^2\,dt\big]\\
& = e^{-K''t}\bigg[\Big(2\langle L_t(u_t)-L_t(\bar{u}_t)+f_t(u_t, \nabla u_t)-f_t(\bar{u}_t, \nabla \bar{u}_t),u_t-\bar{u}_t\rangle \\
& \quad +\sum_{k\in\mathbb{N}}|M_t^k(u_t) - M_t^k(\bar{u}_t)|_{L^2}^2 -K''|u_t-\bar{u}_t|_{L^2}^2\Big)\,dt  \\
& \quad+ \sum_{k\in\mathbb{N}} 2\big(u_t-\bar{u}_t,M_t^k(u_t)-M_t^k(\bar{u}_t)\big) dW_t^k \bigg]
\end{split}     
\end{equation}
almost surely for all $ t \in [0,T]$. 
Substituting \eqref{eq:lipschitz} in \eqref{eq:ito_product} and using Remark \ref{rem:coercivity}, we get 
\begin{equation*}
e^{-K''t}|u_t-\bar{u}_t|_{L^2}^2 \leq 2\sum_{k\in\mathbb{N}}\int_0^t e^{-K''s}\big( u_s-\bar{u}_s,M_s^k(u_s)-M_s^k(\bar{u}_s)\big) dW_s^k 
\end{equation*}
implying that right hand side is a non-negative local martingale (and thus a super-martingale) starting from $0$ and hence for all $ t \in [0,T]$,
\begin{equation*}
\mathbb{E}[e^{-K't}|u_{t}-\bar{u}_{t}|_{L^2}^2]\leq 0.
\end{equation*} 
Thus for all $ t \in [0,T]$, we get  $\mathbb{P}(|u_t - \bar u_t|_{L^2}^2=0)=1$ which, along with the continuity of $u-\bar u$ in ${L^2(\mathscr{D})}$, concludes the 
proof.
\end{proof}

Having proved uniqueness we start preparing the proof of Theorem~\ref{thm:sol_bound}.
For $m\in \mathbb{N}$, consider the truncated function
\[
f_t^{m}(x,r,z)=\left\{
                \begin{array}{ll}
                  f_t(x,-m,z)\,\,\, \text{if}\,\,\,r<-m  \\                
                  f_t(x,r,z)\,\,\,\,\,\,\,\, \text{if}\,\,\,-m\leq r \leq m\\
                  f_t(x,m,z)\,\,\,\,\,\, \text{if}\,\,\,r>m,
                \end{array}
              \right.
\]
and the equation
\begin{equation}
\label{eq:cut}
\begin{split}
du_t^{m} & =(L_t u_t^{m}+f_t^{m}(u_t^{m},\nabla u_t^{m})+f_t^0)dt+\sum_{k\in \mathbb{N}}(M_t^ku_t^{m}+g_t^k)dW_t^k, \\ 
u_t^m & =0 \,\,\, \text{on}\,\,\, \partial \mathscr{D}, \quad u_0^m =\phi \,\,\, \text{on}\,\,\, \mathscr{D}.
\end{split}
\end{equation}

For each $m\in \mathbb{N}$, using Assumption A-\ref{ass:f},  $f_t^{m}(x,r,z )$ is bounded 
and hence \eqref{eq:cut} can be viewed as a SPDE  
on the Gelfand triple 
$H_0^1(\mathscr{D}) \hookrightarrow~ L^2(\mathscr{D}) \hookrightarrow~ H^{-1}(\mathscr{D})$
and all the conditions for existence and uniqueness of solution in~\cite{krylov81}
are satisfied. 
Thus \eqref{eq:cut} has a unique $L^2$-solution in the sense of  \cite[Definition 2.2]{krylov81}. 

We now prove an estimate similar to~\eqref{eq:p_bound}
for the solutions of~\eqref{eq:cut}.
We will do this by applying the It\^o formula from Krylov~\cite{krylov13}
similarly to Dareiotis and Gerencs\'er~\cite{mate15}.
To that end we need to consider the functions
\[
\phi_n(r)=\left\{
\begin{array}{lll}
|r|^p & \text{if} & |r|< n  \\                
n^{p-2}\frac{p(p-1)}{2}(|r|-n)^2+pn^{p-1}(|r|-n)+n^p & \text{if} & |r|\geq n.
\end{array}
\right.
\]
We now collect some key properties of these functions.
We see that $\phi_n$ are twice continuously differentiable and 
\[
|\phi_n(x)|\leq C|x|^2,\,\,|\phi_n'(x)|\leq C|x|,\,\,|\phi_n''(x)|\leq C
\]
where $C$ depends on $p$ and $n \in \mathbb{N}$ only. Further, for any $r\in \mathbb{R}$,
\begin{equation}\label{eq:phi_n_limit}
\phi_n(r)\to |r|^p, \,\,\phi_n'(r)\to p|r|^{p-2}r, \,\,\phi_n''(r)\to p(p-1)|r|^{p-2}
\end{equation}
as $n \to \infty$ and
\begin{equation} \label{eq:phi_n}
\phi_n(r)\leq C|r|^p, \,\, \phi_n'(r)\leq C|r|^{p-1}, \,\,\phi_n''(r)\leq C|r|^{p-2}, \,\,
\end{equation}
where $C$ depends on $p$ only. 

\begin{remark}
\label{rem:phi_inequalities} 
For any $r\in \mathbb{R}$ we have	
\begin{enumerate}[(a)]
\item $ |r\phi_n'(r)|\leq p\phi_n(r) $,
\item $|r^2\phi_n''(r)|\leq p(p-1)\phi_n(r)$,
\item $|\phi_n'(r)|^2\leq 4p\phi_n''(r)\phi_n(r)$,
\item $|\phi_n''(r)|^\frac{p}{p-2}\leq [p(p-1)]^\frac{p}{p-2}\phi_n(r) $.
\end{enumerate}
These inequalities along with Young's inequality imply, for any $\epsilon > 0$,
\begin{enumerate}[(i)]
 \item $ |u_s^m\phi_n'(u_s^m)|\leq C\phi_n(u_s^m) $,
 \item $|u_s^m|^2\phi_n''(u_s^m)\leq C\phi_n(u_s^m)$,
 \item $\sum_{i=1}^d \partial_iu_s^m \phi_n'(u_s^m)\leq \epsilon \phi_n''(u_s^m)|\nabla u_s^m|^2+ C \phi_n(u_s^m)$,
 \item  $|f_s^0\phi_n'(u_s^m)|\leq C|f_s^0|[\phi_n''(u_s^m)]^\frac{1}{2}[\phi_n(u_s^m)]^\frac{1}{2}\leq C|f_s^0|^p+C\phi_n(u_s^m)$,
\item  $|f_s^m(u_s^m, \nabla u_s^m)\phi_n'(u_s^m)|\leq C|f_s^m(u_s^m, \nabla u_s^m)|[\phi_n''(u_s^m)]^\frac{1}{2}[\phi_n(u_s^m)]^\frac{1}{2}\\ 
\indent  \leq C|f_s^m(u_s^m, \nabla u_s^m)|^p+C\phi_n(u_s^m)\leq C|f_s(-m, \nabla u_s^m)|^p+C\phi_n(u_s^m)$, 
\item $|g_s|_{\ell^2}^2\phi_n''(u_s^m)
\leq C\phi_n(u_s^m) + C |g_s|_{\ell^2}^p$,
\end{enumerate}
where the last inequality is obtained using H\"older's inequality and $C$ depends only on $d,p$ and $\epsilon$.  
\end{remark}

Using Theorem 3.1 from~\cite{krylov13}, we get that almost surely 
\begin{equation*}
\begin{split}
\int_\mathscr{D}&\phi_n(u_t^m)dx \\
= & \int_\mathscr{D}\phi_n(u_0^m)dx+\sum_{k\in \mathbb{N}} \int_0^t\int_\mathscr{D}\Big(\sum_{i=1}^d\sigma_s^{ik}\partial_iu_s^m+\mu_s^ku_s^m+g_s^k\Big)\phi_n'(u_s^m)
dxdW_s^k \\
& + \int_0^t\int_\mathscr{D} \Big(\sum_{i=1}^d b_s^i \partial_iu_s^m + c_s u_s^m + f_s^m(u_s^m, \nabla u_s^m)+f^0_s\Big)\phi_n'(u_s^m)dxds\\
&  - \int_0^t\int_\mathscr{D}\sum_{i,j=1}^da_s^{ij}\partial_iu_s^m \phi_n''(u_s^m)\partial_ju_s^m dxds\\
&  + \frac{1}{2} \int_0^t\int_\mathscr{D}\sum_{k\in\mathbb{N}}\Big|\sum_{i=1}^d \sigma_s^{ik}\partial_iu_s^m+\mu_s^ku_s^m+g_s^k\Big|^2\phi_n''(u_s^m)dxds, 
\end{split}
\end{equation*}
 for any $t\in[0,T]$ and $n\in \mathbb{N}$.
 Thus using Assumptions A-\ref{ass:boundedness}, A-\ref{ass:parabolicity} and Young's inequality for any $\epsilon>0$, we obtain almost surely
\begin{equation}
\label{eq:est_para}
\begin{split}
\int_\mathscr{D} & \phi_n(u_t^m)dx  \leq \int_\mathscr{D}\phi_n(u_0^m)dx+ \mathscr{M}_t^{n,m} \\
&+ \int_0^t\int_\mathscr{D} \Big(\sum_{i=1}^d b_s^i \partial_iu_s^m + c_s u_s^m + f_s^m(u_s^m, \nabla u_s^m)+f^0_s\Big)\phi_n'(u_s^m)dxds \\
& - \int_0^t\int_\mathscr{D}\kappa |\nabla u_s^m|^2 \phi_n''(u_s^m)dxds\\
& + \int_0^t\int_\mathscr{D} \Big( \epsilon |\nabla u_s^m|^2 + C |u_s|^2+ C
|g_s|_{\ell^2}^2\Big)\phi_n''(u_s^m)\,dxds, 
\end{split}
\end{equation}
for any $t\in[0,T]$ and $n\in \mathbb{N}$. Here the generic constant $C$ depends only on $d,K$ and $\epsilon$ and 
\[
\mathscr{M}_t^{n,m}:=\sum_{k\in \mathbb{N}} \int_0^t\int_\mathscr{D}\Big(\sum_{i=1}^d\sigma_s^{ik}\partial_iu_s^m+\mu_s^ku_s^m+g_s^k\Big)\phi_n'(u_s^m)
dxdW_s^k
\]
is a martingale. 

Further, using Burkholder--Davis--Gundy's inequality, Remark \ref{rem:phi_inequalities}(c) and H\"older's inequality, we see that
\begin{equation*}
\begin{split}
&\mathbb{E} \sup_{0\leq t \leq T} |\mathscr{M}_t^{n,m}| \\
&\leq C\mathbb{E}\Bigg(\int_0^T\sum_k\bigg( \int_\mathscr{D}\Big|\sum_{i=1}^d\sigma_s^{ik}\partial_iu_s^m+\mu_s^ku_s^m+g_s^k\Big|\Big(\phi_n''(u_s^m)\phi_n(u_s^m)\Big)^\frac{1}{2}dx\bigg)^2ds\Bigg)^\frac{1}{2}\\
& \leq C\mathbb{E}\Bigg(\int_0^T\sum_k\bigg( \int_\mathscr{D}\Big|\sum_{i=1}^d\sigma_s^{ik}\partial_iu_s^m+\mu_s^ku_s^m+g_s^k\Big|^2\phi_n''(u_s^m)dx\int_\mathscr{D}\phi_n(u_s^m)dx\bigg)ds\Bigg)^\frac{1}{2}
\end{split}
\end{equation*}
which, using the same steps as before, in particular Remark~\ref{rem:phi_inequalities} points (ii) and (iv), gives
\begin{equation}
\label{eq:BDG}
\begin{split}
&\mathbb{E} \sup_{0\leq t \leq T} |\mathscr{M}_t^{n,m}| \\
&\leq C\mathbb{E}\Bigg(\int_0^T\bigg( \int_\mathscr{D}\Big(|\nabla u_s^m|^2+|u_s^m|^2+|g_s|_{\ell^2}^2\Big)\phi_n''(u_s^m)dx\int_\mathscr{D}\phi_n(u_s^m)dx\bigg)ds\Bigg)^\frac{1}{2}\\
&\leq C\mathbb{E}\Bigg(\sup_{0\leq t\leq T}\int_\mathscr{D}\phi_n(u_t^m)dx\int_0^T \int_\mathscr{D}\Big[|\nabla u_s^m|^2\phi_n''(u_s^m)+\phi_n(u_s^m)
+|g_s|_{\ell^2}^p\Big]dxds\Bigg)^\frac{1}{2} \\
&\leq \frac{1}{2} \mathbb{E}\sup_{0\leq t\leq T}\int_\mathscr{D}\phi_n(u_t^m)dx + C \mathbb{E}\int_0^T \int_\mathscr{D}\Big[|\nabla u_s^m|^2\phi_n''(u_s^m)+\phi_n(u_s^m)+|g_s|_{\ell^2}^p\Big]dxds
\end{split}
\end{equation}

\begin{lemma}
\label{lem:bound_with_m}
If $u^m$ is the solution to~\eqref{eq:cut}, then 
\begin{equation}
\label{eq:bound_with_m}
\begin{split}
\mathbb{E}\sup_{0\leq t\leq T}|u_t^{m}|_{L^p}^p &+ \mathbb{E}\int_0^t\int_\mathscr{D}|\nabla u_s^m|^2|u_s^m|^{p-2}dxds \\
& \leq C\mathbb{E}\Big(|\phi|_{L^p}^p+ C_m+\Vert f^0 \Vert _{L^p}^p+\Vert|g|_{\ell^2}\Vert^p_{L^p}\Big),
\end{split}
\end{equation}
where $C=C(d,K,\kappa, p)$ and $C_m:=\mathbb{E}\int_0^T\int_{\mathscr{D}}(1+|m|)^{\alpha(p-1)}dxds$ are constants. 
\end{lemma}

\begin{proof}
From  ~\eqref{eq:est_para} and Remark~\ref{rem:phi_inequalities}(iv),(v) and Assumption A-\ref{ass:f}, we get 
\begin{equation*}
\begin{split}
\mathbb{E}\int_\mathscr{D}\phi_n&(u_t^m)dx+\frac{\kappa}{2}\mathbb{E}\int_0^t\int_\mathscr{D}|\nabla u_s^m|^2\phi_n''(u_s^m)dxds \leq C\mathbb{E}\int_\mathscr{D}\phi_n(u_0^m)dx + C_m\\
&  + \mathbb{E}\int_0^t\int_\mathscr{D}|f_s^0|^pdxds  + C\mathbb{E}\int_0^t\int_\mathscr{D}|g_s|_{\ell^2}^p dxds+ C\int_0^t\mathbb{E}
\int_\mathscr{D}\phi_n(u_s^m)dxds\\
\leq & C\mathbb{E}\mathcal{K}_t^m+C\int_0^t\mathbb{E}
\int_\mathscr{D}\phi_n(u_s^m)dxds,
\end{split}
\end{equation*}
where $C=C(d,p,K,\epsilon)$ and 
\[
\mathcal{K}_t^m:=\int_\mathscr{D}|\phi|^pdx+ C_m +\int_0^t\int_\mathscr{D}|f_s^0|^pdxds+\int_0^t\int_\mathscr{D}|g_s|_{\ell^2}^p\,dxds.
\]
Applying Gronwall's lemma, we obtain for any $t\in[0,T]$
\begin{equation} 
\label{eq:after_gronwall}
\mathbb{E}\int_\mathscr{D}\phi_n(u_t^m)dx+\mathbb{E}\int_0^t\int_\mathscr{D}|\nabla u_s^m|^2\phi_n''(u_s^m)dxds \leq C\mathbb{E}\mathcal{K}_t^m
\end{equation}
where  $C=C(d,p,K,\kappa,T)$.

Further, taking the supremum over $t\in[0,T]$ in \eqref{eq:est_para}, using the same estimates as given above  and then taking expectation, we get using \eqref{eq:BDG}
\begin{equation*}
\begin{split}
\mathbb{E}&\sup_{0\leq t \leq T}\int_\mathscr{D}\phi_n(u_t^m)dx \\
 \leq & C\mathbb{E}\int_\mathscr{D}\phi_n(u_0^m)dx+ \mathbb{E}\sup_{0\leq t\leq T}\int_0^t\int_\mathscr{D}f_s^m(u_s^m, \nabla u_s^m)\phi_n'(u_s^m)dxds  \\
& +C\mathbb{E}\int_0^T\int_\mathscr{D}|f_s^0|^pdxds+C\mathbb{E}\int_0^T\int_\mathscr{D}|g_s|_{\ell^2}^p\,dxds+C\int_0^T\mathbb{E} \int_\mathscr{D}\phi_n(u_s^m)dxds\\
& + \frac{1}{2} \mathbb{E}\sup_{0\leq t\leq T}\int_\mathscr{D}\phi_n(u_t^m)dx + C\mathbb{E}\int_0^T \int_\mathscr{D}\Big[|\nabla u_s^m|^2\phi_n''(u_s^m)+\phi_n(u_s^m)\Big]dxds \\
\leq & C\mathbb{E}\int_\mathscr{D}\phi_n(u_0^m)dx+ CC_m+C\mathbb{E}\int_0^T\int_\mathscr{D}|f_s^0|^pdxds \\
&  + C\mathbb{E}\int_0^T\int_{\mathscr{D}} \big[|g_s|_{\ell^2}^p + \phi_n(u_s^m)\big]\,dxds\\
&  + \frac{1}{2} \mathbb{E}\sup_{0\leq t\leq T}\int_\mathscr{D}\phi_n(u_t^m)dx +C\mathbb{E}\int_0^T \int_\mathscr{D}|\nabla u_s^m|^2\phi_n''(u_s^m)dxds\\
\leq & C\mathbb{E}\mathcal{K}_T^m+ \frac{1}{2} \mathbb{E}\sup_{0\leq t\leq T}\int_\mathscr{D}\phi_n(u_t^m)dx <\infty
\end{split}
\end{equation*}
where $C$ does not depend on $n$ and $m$. 
Thus, we have
\begin{equation*}\mathbb{E}\sup_{0\leq t \leq T}\int_\mathscr{D}\phi_n(u_t^m)dx+\mathbb{E}\int_0^T\int_\mathscr{D}|\nabla u_s^m|^2\phi_n''(u_s^m)dxds\leq C\mathbb{E}\mathcal{K}_T^m < \infty,
\end{equation*}
where $C=C(d,p,K,\kappa,T)$. Now we let $n\to \infty$ and apply Fatou's lemma to complete the proof.
\end{proof}

We can now use Lemma~\ref{lem:bound_with_m} and the 
monotonicity of $r\mapsto f^m_t(x,r,z)$
to obtain an estimate for $u^m_t$, where the right-hand-side no longer
depends on $m$.
Let
\[
\mathscr{K}_t:=\int_\mathscr{D}|\phi|^p dx
+ \int_0^t\int_\mathscr{D}\big[|f_s^0|^p + |g_s|_{\ell^2}^p\big]\,dxds.
\]

\begin{lemma} \label{lem:cut bound} 
If $u^m$ is the solution to~\eqref{eq:cut} then there is $C=C(d,p,K,\kappa,T)$
such that
\begin{equation} \label{eq:m,p_bound}
\mathbb{E}\sup_{0\leq t \leq T}|u_t^m|_{L^p}^p +\mathbb{E}\int_0^T\int_\mathscr{D}|\nabla u_s^m|^2|u_s^m|^{p-2}\,dxds\leq C\mathbb{E}\mathscr{K}_T.
\end{equation}	
\end{lemma}

\begin{proof}

From  ~\eqref{eq:est_para} and Remark~\ref{rem:phi_inequalities}(iv), we get 
\begin{equation*}
\begin{split}
\mathbb{E}\int_\mathscr{D}\phi_n(u_t^m)& dx+\frac{\kappa}{2}\mathbb{E}\int_0^t\int_\mathscr{D}|\nabla u_s^m|^2\phi_n''(u_s^m)dxds \\
\leq & C\mathbb{E}\int_\mathscr{D}\phi_n(u_0^m)dx+ \mathbb{E}\int_0^t\int_\mathscr{D}\big[f_s^m(u_s^m, \nabla u_s^m)\phi_n'(u_s^m) + |f_s^0|^p\big]\,dxds \\
&  +C\mathbb{E}\int_0^t\int_\mathscr{D}\big[|g_s|_{\ell^2}^p + \phi_n(u_s^m)\big]\,dxds,
\end{split}
\end{equation*}
where $C=C(d,p,K,\kappa)$.

Taking limit $n\to\infty$ and using Lebesgue's dominated convergence theorem in view of \eqref{eq:bound_with_m}, \eqref{eq:phi_n_limit} and \eqref{eq:phi_n}, we get
\begin{equation}
\label{eq:beforeGronwall}
\begin{split}
\mathbb{E}&\int_\mathscr{D}|u_t^m|^pdx+p(p-1)\frac{\kappa}{2}\mathbb{E}\int_0^t\int_\mathscr{D}|\nabla u_s^m|^2|u_s^m|^{p-2}dxds \\
\leq & C\mathbb{E}\mathscr{K}_t+ p\mathbb{E}\int_0^t\int_\mathscr{D}|u_s^m|^{p-2}f_s^m(u_s^m, \nabla u_s^m)u_s^mdxds+C \mathbb{E} \int_0^t
\int_\mathscr{D}|u_s^m|^p dxds. 
\end{split}
\end{equation}
Using the fact $rf_t^m(r,0)\leq 0$ for any $r\in\mathbb{R},m\in\mathbb{N}, t\in[0,T]$, Young's inequality and Assumption A-\ref{ass:f}, we get
\begin{equation*}
\begin{split}
&p\mathbb{E}\int_0^t\int_\mathscr{D}|u_s^m|^{p-2}f_s^m(u_s^m, \nabla u_s^m)u_s^mdxds \\
&= p\mathbb{E}\int_0^t\int_\mathscr{D}|u_s^m|^{p-2} \big[f_s^m(u_s^m, \nabla u_s^m)-f_s^m(u_s^m,0)+f_s^m(u_s^m,0)\big]u_s^mdxds \\
& \leq \mathbb{E}\int_0^t\int_\mathscr{D}|u_s^m|^{p-2} \big[\frac{\kappa}{4}  |f_s^m(u_s^m, \nabla u_s^m)-f_s^m(u_s^m,0)|^2 +C|u_s^m|^2\big]dxds \\
& \leq \frac{\kappa}{4} \mathbb{E}\int_0^t\int_\mathscr{D}|u_s^m|^{p-2} | \nabla u_s^m|^2 dxds + C \mathbb{E}\int_0^t\int_\mathscr{D}|u_s^m|^pdxds
\end{split}
\end{equation*}
Substituting this in \eqref{eq:beforeGronwall} 
and then applying Gronwall's lemma, we obtain for any $t\in[0,T]$
\begin{equation*} 
\mathbb{E}\int_\mathscr{D}|u_t^m|^pdx+\mathbb{E}\int_0^t\int_\mathscr{D}|\nabla u_s^m|^2|u_s^m|^{p-2}dxds\leq C\mathbb{E}\mathscr{K}_t
\end{equation*}
where  $C=C(d,p,K,\kappa,T)$.

Further, taking the supremum over $t\in[0,T]$ in \eqref{eq:est_para}, using the same estimates as given above and then taking expectation, we get using \eqref{eq:BDG}
\begin{equation*}
\begin{split}
\mathbb{E}&\sup_{0\leq t \leq T}\int_\mathscr{D}\phi_n(u_t^m)dx \\
\leq &  C\mathbb{E}\int_\mathscr{D}\phi_n(u_0^m)dx+ \mathbb{E}\sup_{0\leq t\leq T}\int_0^t\int_\mathscr{D}f_s^m(u_s^m, \nabla u_s^m)\phi_n'(u_s^m)dxds  \\
& + C\mathbb{E}\int_0^T\int_\mathscr{D}\big[|f_s^0|^p + |g_s|_{\ell^2}^p + \phi_n(u_s^m)\big]\,dxds\\
& + \frac{1}{2} \mathbb{E}\sup_{0\leq t\leq T}\int_\mathscr{D}\phi_n(u_t^m)dx + C \mathbb{E}\int_0^T \int_\mathscr{D}|\nabla u_s^m|^2\phi_n''(u_s^m)dxds,
\end{split}
\end{equation*}
where $C$ does not depend on $n$ and $m$. 
Taking limit $n\to\infty$ using Lebesgue's dominated convergence theorem and using \eqref{eq:after_gronwall} along with the 
steps as above, we get
\[
\mathbb{E}\sup_{0\leq t \leq T}\int_\mathscr{D}|u_t^m|^pdx\leq C\mathbb{E}\mathscr{K}_T+ \frac{1}{2} \mathbb{E}\sup_{0\leq t\leq T}\int_\mathscr{D}|u_t^m|^pdx 
\]
and hence the lemma.
\end{proof}

To complete the proof of Theorem~\ref{thm:sol_bound} we need to 
take the limit, as $m\to \infty$ in~\eqref{eq:m,p_bound} and  
to show that~\eqref{eq:spde} has a solution.
To that end we obtain the following result.
\begin{lemma}\label{lem:weak limits}
There is a subsequence of $(m)$ denoted by $(m')$ and
an adapted process $u$ such that 
$u \in L^\alpha(\Omega \times (0,T),\mathscr{P};L^\alpha(\mathscr{D})) \cap L^2(\Omega \times (0,T),\mathscr{P};H_0^1(\mathscr{D}))$ and almost surely $u\in C([0,T]; L^2(\mathscr{D}))$.
Moreover, there exists $f'\in L^\frac{\alpha}{\alpha-1}\big(\Omega \times (0,T),\mathscr{P};L^\frac{\alpha}{\alpha-1}(\mathscr{D})\big)$ such that  
\begin{equation*}
\begin{split}
u^{m'}\rightharpoonup u\quad  \text{in}\quad L^\alpha(\Omega \times (0,T),\mathscr{P};L^\alpha(\mathscr{D})) \cap L^2(\Omega \times (0,T),\mathscr{P};H_0^1(\mathscr{D})),
\end{split}
\end{equation*}
\begin{equation*}
f^{m'}(u^{m'}, \nabla u^{m'})\rightharpoonup f' \quad  \text{in}\quad L^\frac{\alpha}{\alpha-1}\big(\Omega \times (0,T),\mathscr{P};L^\frac{\alpha}{\alpha-1}(\mathscr{D})\big),
\end{equation*}
\begin{equation*}
\begin{split}
L(u^{m'})\rightharpoonup L(u)\quad  \text{in}\quad L^2\big(\Omega \times (0,T),\mathscr{P};H^{-1}(\mathscr{D})\big),\\
M(u^{m'})\rightharpoonup M(u) \quad \text{in}\quad L^2\big(\Omega \times (0,T),\mathscr{P};\ell^2(L^2(\mathscr{D}))\big).	
\end{split}
\end{equation*}
Finally for all $t \in [0,T]$, 
\[
u_t=u_0+\int_0^t  (L_s u_s+f'_s+f_s^0)ds +\sum_{k\in\mathbb{N}}\int_0^t (M_s^k u_s+g_s^k)dW_s^k\,\,a.s. 
\] 
and
\[
\begin{split}
|u_t|_{L^2}^2 = & |\psi|_{L^2}^2 
+ 2\int_0^t \langle L_s u_s+f_s^0, u_s \rangle \,ds 
+ 2 \int_0^t \langle f'_s, u_s \rangle\,ds \\
 & + 2 \sum_{k\in\mathbb{N}} \int_0^t (M_s^k u_s+g_s^k, u_s) \,dW^k_s
+ \sum_{k\in\mathbb{N}} \int_0^t |M_s^k u_s+g_s^k|_{L^2}^2\, ds\,.
\end{split}
\]
\end{lemma}

\begin{proof}

By Lemma \ref{lem:cut bound}, we have $u^{m}\in L^\alpha(\Omega \times (0,T),\mathscr{P};L^\alpha(\mathscr{D}))\cap L^2(\Omega \times (0,T),\mathscr{P};H_0^1(\mathscr{D}))$.
Moreover, using Assumption A-\ref{ass:f} and \eqref{eq:m,p_bound}, we have 

\begin{equation}
\label{eq:bound_f}
\begin{split}
\mathbb{E} \int_0^T \int_{\mathscr{D}}|f^m_t(u^m_t(x), \nabla u^m_t(x))|^\frac{\alpha}{\alpha-1}&\,dxdt \leq K \mathbb{E} \int_0^T \int_{\mathscr{D}}(1+|u_t^m(x)|)^\alpha dxdt \\
& \leq C+C \mathbb{E}\sup_{0\leq t \leq T}\int_\mathscr{D}|u_t^m(x)|^\alpha dx <\infty.
\end{split}
\end{equation}
Thus, $f^m(u^m,\nabla u^m)\in L^\frac{\alpha}{\alpha-1}\big(\Omega \times (0,T),\mathscr{P};L^\frac{\alpha}{\alpha-1}(\mathscr{D})\big)$ 
such that \eqref{eq:m,p_bound} 
and \eqref{eq:bound_f} holds for each $m \in \mathbb{N}$ with a constant independent of $m$.
Since these Banach spaces are reflexive, there exists a subsequence  $(m')$ (see, e.g., Theorem 3.18 in \cite{brezis10}) such that 
\begin{eqnarray*}
 u^{m'}\rightharpoonup v & \text{in} &  L^\alpha(\Omega \times (0,T),\mathscr{P};L^\alpha(\mathscr{D}))\,,\\
u^{m'}\rightharpoonup \bar{v}   & \text{in} & L^2(\Omega \times (0,T),\mathscr{P};H_0^1(\mathscr{D}))\,\,\text{and}\\
f^{m'}(u^{m'},\nabla u^{m'})\rightharpoonup f' & \text{in} &  L^\frac{\alpha}{\alpha-1}\big(\Omega \times (0,T),\mathscr{P};L^\frac{\alpha}{\alpha-1}(\mathscr{D})\big)\,.
\end{eqnarray*}
Moreover, the operators $L$ and $M$ are bounded and linear and hence map a weakly convergent sequence to a weakly convergent sequence. Thus, we have
\begin{eqnarray*}
L(u^{m'})\rightharpoonup L(\bar{v}) & \text{in} & L^2\big(\Omega \times (0,T),\mathscr{P};H^{-1}(\mathscr{D})\big)\,\,\text{and}\\
M(u^{m'})\rightharpoonup M(\bar{v}) & \text{in} & L^2\big(\Omega \times (0,T),\mathscr{P};\ell^2(L^2(\mathscr{D}))\big)\,.
\end{eqnarray*}
 Note that for any adapted and bounded real valued process $\eta_t$ and $\xi \in C_0^\infty(\mathscr{D})$, we have
 \[
 \mathbb{E}\int_0^T\eta_t \langle v_t-\bar{v}_t,\xi \rangle dt=\mathbb{E}\int_0^T\eta_t \langle v_t-u_t^{m'},\xi \rangle dt+\mathbb{E}\int_0^T\eta_t \langle u_t^{m'}-\bar{v}_t,\xi \rangle dt \to 0
 \]
as $m' \to \infty$. Since $C_0^\infty(\mathscr{D})$ is dense in $L^\alpha(\mathscr{D})$ and $H_0^1(\mathscr{D})$, we have the processes $v$ and $\bar{v}$ are equal $dt\times \mathbb{P}$~ almost everywhere. Further, the Bochner integral and the stochastic integral are bounded linear operators and hence are continuous with respect to weak topologies. Again, we have
\begin{equation*}
\begin{split}
&\mathbb{E}\int_0^T\eta_t(u_t^{m'},\xi)dt \\
&=\mathbb{E}\int_0^T\eta_t\Big((u_0^{m'},\xi)+\int_0^t\langle L_su_s^{m'}+f_s^{m'}+f_s^0,\xi\rangle ds +\sum_{k\in\mathbb{N}}\int_0^t(\xi ,M_s^ku_s^{m'}+g_s^k)dW_s^k \Big)dt.
\end{split}
\end{equation*}
On taking limit $m' \to \infty$, we get
\begin{equation*}
\begin{split}
&\mathbb{E}\int_0^T\eta_t(v_t,\xi)dt \\
&=\mathbb{E}\int_0^T\eta_t\Big((u_0,\xi)+\int_0^t\langle L_sv_s+f'_s+f_s^0,\xi\rangle ds +\sum_{k\in\mathbb{N}}\int_0^t(\xi ,M_s^k\textcolor{red}{v}_s+g_s^k)dW_s^k \Big)dt
\end{split}
\end{equation*}
for any adapted and bounded real valued process $\eta_t$ and $\xi \in C_0^\infty(\mathscr{D})$. Since $C_0^\infty(\mathscr{D})$ is dense in $L^\alpha(\mathscr{D})$ and $H_0^1(\mathscr{D})$, we have
\[
v_t=u_0+\int_0^t  (L_sv_s+f'_s+f_s^0)ds +\sum_{k\in\mathbb{N}}\int_0^t (M_s^kv_s+g_s^k)dW_s^k 
\] 
 $dt\times \mathbb{P}$~ almost everywhere. 
Using It\^o formula for processes taking values in intersection of Banach spaces from Gy\"ongy and \v{S}i\v{s}ka \cite{gyongy17}, there exists an $L^2(\mathscr{D})$-valued continuous modification $u$ of $v$  which satisfies above equality almost surely for all $t\in[0,T]$. 
\end{proof}


\begin{remark}\label{rem:strong convergence}
For  $\psi \in L^\alpha(\Omega \times (0,T),\mathscr{P};L^\alpha(\mathscr{D}))\cap L^2(\Omega \times (0,T),\mathscr{P};H_0^1(\mathscr{D})), $ we have 
\[
f^{m'}(\psi,\nabla \psi)\to f(\psi, \nabla \psi)
\]
in $L^\frac{\alpha}{\alpha-1}(\Omega \times (0,T),\mathscr{P};L^\frac{\alpha}{\alpha-1}(\mathscr{D}))$. Indeed, by definition of $f^{m'}$, as $m'\to\infty$
\[
 f_s^{m'}(\psi_s(x),\nabla \psi_s(x))\to f_s(\psi_s(x),\nabla \psi_s(x))\,\,\,
 \forall \omega, s, x\,.
\]
Moreover $|f_s^{m'}(r,z)|\leq |f_s(r,z)|$ and due to Assumption A-\ref{ass:f},
\begin{equation*}
\mathbb{E}\int_0^T|f_s(\psi_s, \nabla \psi_s(x))|_{L^\frac{\alpha}{\alpha-1}}^\frac{\alpha}{\alpha-1}ds
\leq C\mathbb{E}\int_0^T\int_\mathscr{D}\Big(1+|\psi_s(x)|^\alpha \Big) dxds <\infty.
\end{equation*}
Therefore we may use Lebesgue Dominated Convergence Theorem to obtain
\begin{equation*}
\begin{split}
\lim_{m' \to \infty }&\mathbb{E}\int_0^T\int_\mathscr{D}|f_s^{m'}(\psi_s(x),\nabla \psi_s(x))-f_s(\psi_s(x),\nabla \psi_s(x))|^\frac{\alpha}{\alpha-1}dxds \\
&=\mathbb{E}\int_0^T\int_\mathscr{D}\lim_{m' \to \infty }|f_s^{m'}(\psi_s(x),\nabla \psi_s(x))-f_s(\psi_s(x),\nabla \psi_s(x))|^\frac{\alpha}{\alpha-1}dxds=0.
\end{split}
\end{equation*}
\end{remark}

\begin{proof}[Proof of Theorem~\ref{thm:sol_bound}]

In order to show the weak limit $u$ obtained in Lemma \ref{lem:weak limits} is indeed the unique solution of SPDE \eqref{eq:spde}, it remains to show that $f'=f(u,\nabla u)$ which can be shown using the monotonicity argument as below.
 
Define for each $w\in L^\alpha(\mathscr{D})\cap H_0^1(\mathscr{D}), s\in(0,T)$ and $k\in\mathbb{N}$, the operators
\[
A_sw:=L_sw+f_s^0 \quad \text{and} \quad B_s^kw:=M_s^kw+g_s^k.
\]
Then for any $w,w' \in L^\alpha(\mathscr{D})\cap H_0^1(\mathscr{D})$, we have using Remark \ref{rem:coercivity}
\begin{equation}\label{eq:monotonicity}
2\langle A_sw-A_sw',w-w' \rangle + \sum_{k\in\mathbb{N}} |B_s^kw-B_s^kw'|_{L^2}^2 \leq -\kappa |w-w'|^2_{H_0^1}+K'|w-w'|^2_{L^2}.
\end{equation}
Consider $\psi \in L^\alpha(\Omega\times (0,T),\mathscr{P};L^\alpha(\mathscr{D})) \cap L^2(\Omega\times (0,T),\mathscr{P};H_0^1(\mathscr{D}))$. Then using Assumption A-\ref{ass:f}, Remark \ref{rem:decrease} and definition of $f^m$, we have 
\begin{equation}\label{eq:f'decrease}
\langle f^{m'}_s(u_s^{m'}, \nabla u_s^{m'})-f^{m'}_s(\psi_s,\nabla u_s^{m'}),u_s^{m'}-\psi_s\rangle \leq 0
\end{equation}
almost surely for all $s\in[0,T]$. Moreover using Young's inequality and Assumption A-\ref{ass:f}, we have almost surely for all $s\in[0,T]$
\begin{equation}\label{eq:f'monotone}
2\langle f^{m'}_s(\psi_s,\nabla u_s^{m'})-f^{m'}_s(\psi_s,\nabla \psi_s),u_s^{m'}-\psi_s\rangle \leq \kappa|\nabla(u_s^{m'}-\psi_s)|^2_{L^2}+C|u_s^{m'}-\psi_s|^2_{L^2}.
\end{equation}
Define $K'':=K'+C$, where $K'$ and $C$ are as in \eqref{eq:monotonicity} and \eqref{eq:f'monotone} above. Then using the product rule and It\^{o}'s formula, we obtain 
\begin{equation}
\label{eq:ito1}
\begin{split}
\mathbb{E}  \big(e^{-K''t}&|u_t|_{L^2}^2\big) 
-  \mathbb{E}(|u_0|_{L^2}^2)   \\
& = \mathbb{E}\Big[\int_0^te^{-K''s}\Big(2\langle A_su_s+f'_s,u_s\rangle  
+\sum_{k\in\mathbb{N}} |B^k_su_s|_{L^2}^2-K''|u_s|_{L^2}^2\Big)ds\Big] 
\end{split}
\end{equation}
and
\begin{equation}
\label{eq:ito2}
\begin{split}
\mathbb{E}\big(e^{-K''t}|u_t^{m'}|_{L^2}^2\big)
-\mathbb{E}(|u_0^{m'}|_{L^2}^2)   =  \mathbb{E}\Big[& \int_0^te^{-K''s}\Big(2\langle A_su_s^{m'}+f_s^{m'}(u_s^{m'},\nabla u_s^{m'} ),u_s^{m'}\rangle \\
& +\sum_{k\in\mathbb{N}} |B_s^ku_s^{m'}|_{L^2}^2-K''|u_s^{m'}|_{L^2}^2\Big)ds\Big] 
\end{split}
\end{equation}
for all $t\in [0,T]$.

We now need to re-arrange the right-hand side of \eqref{eq:ito2} so that we can use the monotonicity assumptions. We have 
\begin{equation}\label{eq:to_use_mono}
\begin{split}
\mathbb{E}&\Big[\int_0^te^{-K''s}\Big(2\langle A_su_s^{m'}+f_s^{m'}(u_s^{m'}, \nabla u_s^{m'}),u_s^{m'}\rangle +\sum_{k\in\mathbb{N}} |B_s^ku_s^{m'}|_{L^2}^2-K''|u_s^{m'}|_{L^2}^2\Big)ds\Big]  \\
=&\mathbb{E}\Big[\int_0^te^{-K''s}\Big(2\langle A_su_s^{m'}-A_s\psi_s,u_s^{m'}\rangle+2\langle A_s\psi_s,u_s^{m'}\rangle +2 \langle A_su_s^{m'}-A_s\psi_s,\psi_s\rangle\\
&+ 2\langle f^{m'}_s(u_s^{m'}, \nabla u_s^{m'})-f^{m'}_s(\psi_s,\nabla \psi_s),u_s^{m'}-\psi_s\rangle+2\langle f^{m'}_s(\psi_s,\nabla \psi_s),u_s^{m'} \rangle \\
& +2 \langle f_s^{m'}(u_s^{m'},\nabla u_s^{m'})-f_s^{m'}(\psi_s,\nabla \psi_s),\psi_s\rangle +\sum_{k\in\mathbb{N}} \big|B_s^ku_s^{m'}-B_s^k\psi_s\big|_{L^2}^2-\sum_{k\in\mathbb{N}}|B_s^k\psi_s|_{L^2}^2  \\
&  +2\sum_{k\in\mathbb{N}}\big(B_s^ku_s^{m'},B_s^k\psi_s\big)-K''\left[|u_s^{m'}-\psi_s|_{L^2}^2-|\psi_s|_{L^2}^2 
+2(u_s^{m'},\psi_s)\right]\!\Big)ds\Big] \,. 
\end{split}
\end{equation}
Using \eqref{eq:f'decrease} and \eqref{eq:f'monotone}, we have
\begin{equation*}
\begin{split}
&2\langle f^{m'}_s(u_s^{m'}, \nabla u_s^{m'})-f^{m'}_s(\psi_s,\nabla \psi_s),u_s^{m'}-\psi_s\rangle \\
&=2\langle f^{m'}_s(u_s^{m'}, \nabla u_s^{m'})-f^{m'}_s(\psi_s,\nabla u_s^{m'})+f^{m'}_s(\psi_s,\nabla u_s^{m'})-f^{m'}_s(\psi_s,\nabla \psi_s),u_s^{m'}-\psi_s\rangle \\
& \leq \kappa|\nabla(u_s^{m'}-\psi_s)|^2_{L^2}+C|u_s^{m'}-\psi_s|^2_{L^2}
\end{split}
\end{equation*}
and hence using \eqref{eq:monotonicity} in \eqref{eq:to_use_mono} together with \eqref{eq:ito2}, we obtain for all $t\in[0,T]$
\begin{equation*}
\begin{split}
 \mathbb{E} & \big(e^{-K''t}|u_t^{m'}|_{L^2}^2\big)-\mathbb{E}(|u_0^{m'}|_{L^2}^2)  \\
\leq & \mathbb{E}\Big[\int_0^te^{-K''s}\Big( 2\langle A_s\psi_s,u_s^{m'}\rangle +2 \langle A_su_s^{m'}-A_s\psi_s,\psi_s\rangle \\
& +2\langle f_s^{m'}(\psi_s,\nabla \psi_s),u_s^{m'}\rangle+2 \langle f_s^{m'}(u_s^{m'},\nabla u_s^{m'})-f_s^{m'}(\psi_s,\nabla \psi_s),\psi_s\rangle 
\\
&   -\sum_{k\in\mathbb{N}}|B_s^k\psi_s|_{L^2}^2 +2\sum_{k\in \mathbb{N}}\big(B_s^ku_s^{m'},B_s^k\psi_s\big)+K''\big[|\psi_s|_{L^2}^2 
 -2(u_s^{m'},\psi_s)\big] \Big)ds\Big]. 
 \end{split}
\end{equation*}
Now, integrating over $t$ from $0$ to $T$, letting $m' \rightarrow \infty$ and using the weak lower semicontinuity of the norm, we obtain
\begin{equation}
\label{eq:weaklimits}
\begin{split}
\mathbb{E}\Big[\int_0^T &\big(e^{-K''t}|u_t|_{L^2}^2-|u_0|_{L^2}^2\big)dt\Big] \\
\leq &  \liminf_{m'\rightarrow\infty}\mathbb{E}\Big[\int_0^T\big(e^{-K''t}|u_t^{m'}|_{L^2}^2-|u_0^{m'}|_{L^2}^2\big)dt\Big]  \\
\leq  & \mathbb{E}\Big[\int_0^T\int_0^te^{-K''s}\Big( 2\langle A_s\psi_s,u_s\rangle+2 \langle A_su_s-A_s\psi_s,\psi_s\rangle   \\
&  +2\langle f_s(\psi_s, \nabla \psi_s),u_s\rangle+2 \langle f'_s-f_s(\psi_s,\nabla \psi_s),\psi_s\rangle  -\sum_{k\in\mathbb{N}}|B_s^k\psi_s|_{L^2}^2 \\
& +2\sum_{k\in\mathbb{N}}(B^k_su_s,B_s^k(\psi_s))+K''\left[|\psi_s|_{L^2}^2 
 -2(u_s,\psi_s)\right] \Big)dsdt\Big]
\end{split}
\end{equation}
where we have used Remark \ref{rem:strong convergence} in last inequality.
Again, integrating from $0$ to $T$ in~\eqref{eq:ito1} and combining this 
with \eqref{eq:weaklimits}, we get
\begin{equation*}
\begin{split}
\mathbb{E}\Big[\int_0^T&\int_0^t e^{-K''s}\Big( 2\langle A_su_s-A_s\psi_s,u_s-\psi_s \rangle +2\langle f'_s-f_s(\psi_s,\nabla \psi_s),u_s-\psi_s \rangle \\
&\quad+\sum_{k\in\mathbb{N}}|B_s^k\psi_s-B^k_su_s|_{L^2}^2 -K''|u_s-\psi_s|_{L^2}^2 \Big)dsdt\Big]\leq 0
\end{split}
\end{equation*}
which on using \eqref{eq:monotonicity} gives
\begin{equation}
\label{eq:limit f}
\mathbb{E}\Big[\int_0^T\int_0^t e^{-K''s}\Big( 2\langle f'_s-f_s(\psi_s, \nabla \psi_s),u_s-\psi_s \rangle  
 \Big)dsdt\Big]\leq 0. 
\end{equation}

Let $\eta\in L^\infty((0,T)\times \Omega;\mathbb{R})$, $\phi\in C_0^\infty(\mathscr{D})$, 
$\epsilon\in (0,1)$ and let $\psi=u-\epsilon \eta \phi$. Then from~\eqref{eq:limit f} one obtains that
\begin{equation*}
\mathbb{E} \Big[ \int_0^T  \int_0^t 2\epsilon e^{-K''s}  \langle f'_s-f_s(u_s-\epsilon \eta_s \phi,\nabla u_s-\epsilon \eta_s \nabla \phi),\eta_s \phi \rangle  dsdt\Big]\leq 0. 
\end{equation*}

Dividing by $\epsilon$, letting $\epsilon\rightarrow  0$, 
using Lebesgue dominated convergence theorem and 
Assumption A-\ref{ass:f} leads to
\begin{align}
&\mathbb{E}\Big[\int_0^T \int_0^t 2 e^{-K''s}
 \eta_s \langle f'_s-f_s(u_s,\nabla u_s),\phi \rangle dsdt\Big]\leq 0. \notag
\end{align}
Since this holds for any $\eta\in L^\infty((0,T)\times \Omega, \mathscr{P} ;\mathbb{R})$ and $\phi\in C_0^\infty(\mathscr{D})$, one gets that $f(u,\nabla u)= ~f'$ which concludes the proof.

 Further, taking $m \to \infty$ in \eqref{eq:m,p_bound} and using the weak lower semicontinuity of the norm, we obtain the following estimates for the solution of \eqref{eq:spde}

\begin{equation*}  
\begin{split}
\mathbb{E}\sup_{0\leq t \leq T}|u_t|_{L^p}^p 
+ & \mathbb{E}\int_0^T\int_\mathscr{D}|\nabla u_s|^2|u_s|^{p-2}dxds \\
&\leq \liminf_{m\to \infty} \Big[\mathbb{E}\sup_{0\leq t \leq T}|u_t^m|_{L^p}^p+\mathbb{E}\int_0^T\int_\mathscr{D}|\nabla u_s^m|^2|u_s^m|^{p-2}dxds\Big]\\
& \leq C\mathbb{E}\Big(|\phi|_{L^p}^p+\Vert f^0 \Vert _{L^p}^p+\Vert|g|_{\ell^2}\Vert^p_{L^p}\Big).
\end{split}
\end{equation*} 

\end{proof}

\section{Interior Regularity}
\label{sec:regularity}
In this section, we present the results on interior regularity of the solution to SPDE \eqref{eq:spde}. The main result is stated in Theorem \ref{thm:higher_regularity}. The idea is to prove the result for the linear SPDE first and then use it along with the $L^p$-estimates obtained in Section \ref{sec: Lp} to prove Theorem \ref{thm:higher_regularity}. We do not claim the result for the linear case to be new, however we could not find such result in literature in sufficient generality. 

To raise the regularity of the solution one needs the given data to be sufficiently smooth. Thus, we assume the following condition on the coefficients before stating the main result of this section. 
\begin{assumption} \label{ass:der_boundedness}
For any $i,j=1,\ldots,d$, the coefficients $a^{ij}, b^i$ and $c$ 
and their spatial derivatives up to order $n$ are real-valued, 
$\mathscr{P} \times \mathscr{B}(\mathscr{D})$-measurable and are bounded by $K$. 
The coefficients $\sigma^i=(\sigma^{ik})_{k=1}^\infty$, $\mu=(\mu^k)_{k=1}^\infty$ 
and their spatial derivatives up to order $n$  are $\ell^2$-valued,
$\mathscr{P} \times \mathscr{B}(\mathscr{D})$-measurable and almost surely
\[
\sum_{i=1}^d\sum_{k\in \mathbb{N}} \sum_{|\gamma|\leq n}|D^\gamma\sigma_t^{ik}(x)|^2+\sum_{k\in\mathbb{N}}\sum_{|\gamma|\leq n}|D^\gamma\mu_t^k(x)|^2 \leq K
\] 
for all $t$ and $x$.
\end{assumption}

For $A$, $B$ subsets of $\mathbb R^d$ let $\dist(A, B)$ denote the distance between $A$ and $B$. 
Further, for $\ell = 1,2$ define
\begin{equation*}
\begin{split}
\mathcal{I}^\ell:=\mathbb{E}\Big[ \sum_{|\gamma|\leq \ell}|D^\gamma \phi|^2_{L^2}+ & \sum_{|\gamma|\leq \ell-1}\Vert D^\gamma f^0 \Vert_{L^2}^2 +\sum_{|\gamma|\leq \ell}\Vert|D^\gamma g|_{\ell^2}\Vert^2_{L^2} \\
 & +|\phi|_{L^{2\alpha-2}}^{2\alpha-2}+\Vert f^0 \Vert _{L^{2\alpha-2}}^{2\alpha-2}+\Vert|g|_{\ell^2}\Vert^{2\alpha-2}_{L^{2\alpha-2}} \Big]\,.
\end{split}
\end{equation*}

\begin{theorem}\label{thm:higher_regularity}
Let Assumptions A-\ref{ass:parabolicity} to A-\ref{ass:initial} hold and $u$ be the solution to \eqref{eq:spde}. Fix some open $\mathscr{D''}\Subset\mathscr{D'}\Subset \mathscr{D}$ such that $\dist(\mathscr{D'}, \partial \mathscr D)<1$ and  $\dist(\mathscr{D''}, \partial \mathscr D')<1$.
\begin{enumerate} [(i)]
\item
If Assumption A-\ref{ass:der_boundedness} holds with $n=1$, 
and if $\phi \in L^2(\Omega,\mathscr{F}_0;H^1(\mathscr{D}))$ and $g \in~ L^2( \Omega \times (0,T),\mathscr{P};H^1(\mathscr{D};\ell^2))$, then 
\[
u \in C([0,T], H^1(\mathscr{D}'))\,\,\,\text{a.s. and} \,\,\,  
u\in L^2(\Omega \times (0,T),\mathscr{P};H^{2}(\mathscr{D}')).
\]
Moreover, there is $C=C(d,T,K,\kappa)$ such that
 \begin{equation}
 \begin{split}\label{eq:blow1_regularity}
 \mathbb{E}\sup_{0\leq t\leq T}|\partial_i u_t|_{L^2(\mathscr{D'})}^2
+  & \mathbb{E}\int_0^T|\partial_i u_t|^2_{H^1(\mathscr{D'})} dt \leq C\dist(\mathscr{D'},\partial \mathscr{D})^{-2} \mathcal{I}^1
\end{split} 
 \end{equation}
 for all $i=1,\ldots,d$.

\item
 Further, in case the semilinear term $f$ does not depend on $z$, if Assumption A-\ref{ass:boundedness} holds with $n=2$, if
$\phi \in L^2(\Omega,\mathscr{F}_0;H^2(\mathscr{D}))$, 
$f^0 \in L^2( \Omega \times (0,T),\mathscr{P};H^1(\mathscr{D}))$ and 
$g \in L^2( \Omega \times (0,T),\mathscr{P};H^2(\mathscr{D};\ell^2))$ and if 
almost surely
\begin{equation}\label{eq:f_prime}
|\partial_rf_t(x,r)|\leq K(1+|r|)^{\alpha-2}\,\,\,\text{and}\,\,\,
|\partial_i f_t(x,r)| \leq K(1+|r|)^{\alpha - 1}
\end{equation}
for all $i=1,\dots, d$, $t\in [0,T], x\in \mathscr{D}$ and all $r\in \mathbb{R}$,  then we have 
\[
u\in C([0,T], H^2(\mathscr{D}'')) \,\,\,\text{a.s. and} \,\,\, u\in  L^2(\Omega \times (0,T),\mathscr{P};H^{3}(\mathscr{D}'')).
\] 
Furthermore, there is $C=C(d,T,K,\kappa)$ such that
\begin{equation}\label{eq:blow2_regularity}
\begin{split}
  \mathbb{E}\sup_{0\leq t\leq T}|\partial_i \partial_j u_t|_{L^2(\mathscr{D''})}^2
                                +\mathbb{E}\int_0^T|& \partial_i \partial_j u_t|^2_{H^1(\mathscr{D''})} dt 
  \leq C \dist(\mathscr{D}'',\partial \mathscr D')^{-2} \mathcal{I}^2 \\
  &  + C \dist(\mathscr{D}'',\partial \mathscr D')^{-2} \dist(\mathscr{D}',\partial \mathscr D)^{-2}\mathcal{I}^1
\end{split} 
\end{equation}
for all $i,j=1,\ldots,d$.
\end{enumerate}
\end{theorem}

One can obtain regularity results up to the boundary in appropriate weighted Sobolev spaces using results from Krylov \cite{krylov94} along with the $L^p$-estimates obtained in Theorem~\ref{thm:sol_bound}. However, obtaining the similar results for the linear equations using $L^p$-theory is more useful . We will discuss this in Section \ref{sec:Lp_reg}.  
 
As mentioned before, we will first get the results for linear equations. So,
we consider the following linear stochastic evolution equation:
\begin{equation}
\label{eq:lin_spde}
\begin{split}
dv_t & =(L_tv_t+f_t)dt+\sum_{k\in \mathbb{N}}(M_t^kv_t+g_t^k)dW_t^k
\,\,\,\ \text{on}\,\,\, [0,T]\times \mathscr{D},
 \end{split}
\end{equation}
where the operators $L$ and $M^k$ are defined in \eqref{eq:def_L,M}. As can be seen in what follows, one can raise the regularity to any order for the linear equation by assuming the given data to be sufficiently smooth. Thus we make the following assumption on initial data and the free terms and then state the result in Theorem \ref{thm:H_m_regularity}.
 
Let $n\geq 0$ be an integer. 
 

\begin{assumption}\label{ass:der_initial}
Assume that $v_0 \in L^2(\Omega,\mathscr{F}_0;H^n(\mathscr{D}))$, 
$g \in~L^2( \Omega \times (0,T),\mathscr{P};H^n(\mathscr{D};\ell^2))$ 
and $f\in L^2( \Omega \times (0,T),\mathscr{P};H^{n-1}(\mathscr{D}))$.


\end{assumption}

\begin{theorem}\label{thm:H_m_regularity}
Assume that $v$ is a continuous $L^2(\mathscr{D})$-valued adapted process
such that $v \in~ L^2(\Omega \times (0,T), \mathscr{P}; H^1(\mathscr{D}))$, 
and it satisfies \eqref{eq:lin_spde}. 
If Assumptions  A-~\ref{ass:parabolicity}, A-~\ref{ass:der_boundedness} 
and  A-~\ref{ass:der_initial} hold,
then for all open $\mathscr{D'}\Subset \mathscr{D}$,
\begin{equation*}
v \in C([0,T], H^n(\mathscr{D}')) \,\,\,\text{a.s.} \,\,\, \text{and}\,\,\,
v \in L^2(\Omega \times (0,T),\mathscr{P};H^{n+1}(\mathscr{D}'))
 \end{equation*}
\end{theorem}

We will prove Theorem~\ref{thm:H_m_regularity} via Lemmas \ref{lem:H_2_regularity} and \ref{lem:H_k_regularity}. 
In Lemma~\ref{lem:H_2_regularity}, we first prove the special case $n=1$. 
\begin{lemma}\label{lem:H_2_regularity}
Assume that $v \in C([0,T]; L^2(\mathscr{D}))$ a.s., $v$ is adapted and satisfies ~\eqref{eq:lin_spde} and moreover $v \in L^2(\Omega \times (0,T), \mathscr{P}; H^1(\mathscr{D}))$. 
If Assumptions A-\ref{ass:parabolicity}, A-\ref{ass:der_boundedness}  and  A-\ref{ass:der_initial} hold with $n=1$, then there is $C=C(d,T,K,\kappa)$ such that  
 \begin{equation}
 \begin{split}\label{eq:H_2_regularity}
 \mathbb{E}\sup_{0\leq t\leq T}|\partial_i v_t|_{L^2(\mathscr{D'})}^2
+ & \mathbb{E}\int_0^T|\partial_i v_t|^2_{H^1(\mathscr{D'})} dt 
 \leq C\dist(\mathscr{D'},\partial \mathscr{D})^{-2} \Bigg[\mathbb{E} \int_\mathscr{D}|\nabla v_0|^2 dx \\  
 & + \mathbb{E}\int_0^T\int_\mathscr{D} \Big[|\nabla v_t|^2  
 + |f_t|^2 + |v_t|^2 + \sum_{k\in  
 \mathbb{N}}|\nabla g_t^k|^2 \Big]dxdt\Bigg]
\end{split} 
 \end{equation}
 for all $i=1,\ldots,d$ and all open $\mathscr{D'}\Subset \mathscr{D}$ such that $\dist(\mathscr{D'},\partial \mathscr{D})<1$.
\end{lemma}
\begin{proof}
Let $\zeta = \dist(\mathscr{D'},\partial \mathscr{D})$.
We consider a cut-off function $\eta \in C_0^\infty(\mathscr{D})$ which is $1$ on $\mathscr{D}'$ and such that $\eta \leq 1$ and $|\partial_i\eta|\leq C \zeta^{-1}$ for $i=1,2,\ldots, d$. 
Define the $l^{th}$-difference quotient, $l\in \{1,2,\ldots,d \}$, by
 \[
 \delta_l^hu(x):=\frac{1}{h}\big(T_l^hu-u\big)(x),\qquad x\in\mathbb{R}^d
 \]
 where $T_l^hu(x)=u(x+he_l)$ is the shift operator and the step-size $h$ satisfies $2|h|< \dist(\supp \eta,\partial{\mathscr{D}})$. 
From \eqref{eq:lin_spde}, we get
\[
d(\eta \delta_l^h v_t)=\eta \delta_l^h(L_tv_t+f_t)dt+\eta \sum_{k\in \mathbb{N}}\delta_l^h(M_t^kv_t+g_t^k)dW_t^k.
\] 
 Applying It\^o's formula for the square of $L^2$-norm, we get
 
 \begin{equation*}
 \begin{split}
 d|\eta \delta_l^h v_t|_{L^2(\mathscr{D})}^2 =2 \langle \eta \delta_l^h(L_tv_t+f_t),\eta \delta_l^h v_t \rangle dt &+ 2 \sum_{k\in \mathbb{N}}(\eta \delta_l^h(M_t^kv_t+g_t^k), \eta \delta_l^h v_t)dW_t^k \\
 &+ \sum_{k\in \mathbb{N}}|\eta \delta_l^h(M_t^kv_t+g_t^k)|_{L^2(\mathscr{D})}^2 dt.
 \end{split}
 \end{equation*}
It follows from the definition of $\delta_l^h$ and linearity of $\partial_j$, that the two operators commute. Thus, using integration by parts and the formula 
\begin{equation*}
\label{rem:delta}
 \delta_l^h(vw)(x)=\delta_l^hv(x)T_l^hw(x)+v(x)\delta_l^hw(x)
\end{equation*}
we get,

\begin{equation}
 \begin{split} \label{eq:ito_difference}
 \int_\mathscr{D}& \eta^2 |\delta_l^h v_t|^2 dx 
     =  \int_\mathscr{D}\eta^2 |\delta_l^h v_0|^2 dx
       +2 \int_0^t\int_\mathscr{D}\eta ^2 \delta_l^h(L_sv_s+f_s)\delta_l^h v_s  dxds \\
 &   + \mathscr{M}_t^h  + \sum_{k\in \mathbb{N}}\int_0^t\int_\mathscr{D}\eta^2 
          |\delta_l^h(M_s^kv_s+g_s^k)|^2 dxds \\
= & I_0-2 \int_0^t\int_\mathscr{D}\eta ^2\sum_{i,j=1}^d a_s^{ij} \, \partial_i(\delta_l^h v_s)\, \partial_j(\delta_l^h v_s)+I_1+I_2+I_3+ \mathscr{M}_t^h +I_4
\end{split}
\end{equation}
where,
\begin{equation*}
\begin{split}            
 I_0  := &  \int_\mathscr{D}\eta^2 |\delta_l^h v_0|^2 dx , \\
 I_1 := & -2 \int_0^t\int_\mathscr{D}\eta ^2                 
         \sum_{i,j=1}^d\delta_l^h a_s^{ij}\,\partial_i (T_l^hv_s)\partial_j(\delta_l^h  v_s) dxds , \\
  I_2 := &  -4\int_0^t\int_\mathscr{D}\eta \sum_{i,j=1}^d\big[\delta_l^h a_s^{ij}\, \partial_i   
              (T_l^hv_s)+a_s^{ij} \, \partial_i(\delta_l^h v_s)\big] \partial_j\eta \delta_l^h     
              v_s dxds \\
        I_3 :=  & 2 \int_0^t\int_\mathscr{D}\eta ^2  \Big[\sum_{i=1}^d\{\delta_l^h b_s^i \, \partial_i   
            (T_l^hv_s)+ b_s^i \, \delta_l^h(\partial_i v_s)\}   \\
  &    +\delta_l^h c_s \, T_l^hv_s+c_s \, \delta_l^h v_s+\delta_l^h f_s \Big]\delta_l^h v_sdxds ,\\
 I_4 :=& \sum_{k\in \mathbb{N}}\int_0^t\int_\mathscr{D} \eta^2  \Big|\sum_{i=1}^d\delta_l^h \sigma_s^{ik} \, \partial_i(T_l^hv_s)+ \delta_l^h\mu_s^k \, T_l^hv_s \\
&   +\sum_{i=1}^d \sigma_s^{ik} \, \partial_i(\delta_l^h v_s)+\mu_s^k \,  \delta_l^hv_s+\delta_l^hg_s^k \Big|^2 dxds                                                                                           
\end{split}
\end{equation*}
and
\[
\mathscr{M}_t^h:= 2 \sum_{k\in \mathbb{N}}\int_0^t\int_\mathscr{D}\eta^2 \delta_l^h(M_s^kv_s+g_s^k) \delta_l^h v_sdxdW_s^k.
\] 
Now, we see that 
\begin{equation*}
\begin{split}
& I_4=  \sum_{k\in \mathbb{N}} \int_0^t \int_\mathscr{D} \eta^2 \left[ \Big|\sum_{i=1}^d \delta_l^h    
        \sigma_s^{ik} \, \partial_i(T_l^hv_s)+ \delta_l^h \mu_s^k \, T_l^hv_s \Big|^2 \right. \\
& +2\Big[\sum_{i=1}^d \delta_l^h \sigma_s^{ik} \, \partial_i(T_l^hv_s)
               + \delta_l^h\mu_s^k \, T_l^hv_s \Big]\Big[\sum_{i=1}^d  
               \sigma_s^{ik} \, \partial_i(\delta_l^hv_s)+\mu_s^k \delta_l^hv_s+\delta_l^hg_s^k     
                     \Big]  \\
& \left. +\Big|\sum_{i=1}^d  \sigma_s^{ik} \, \partial_i(\delta_l^h v_s)
               +\mu_s^k \, \delta_l^h v_s + \delta_l^h g_s^k \Big|^2 
               \right]dxds \leq  \sum_{i,j=1}^d \sigma_s^{ik} \, \partial_i(\delta_l^h v_s) \, 
       \sigma_s^{jk} \, \partial_j(\delta_l^h v_s)+ \bar{I_4}
\end{split}
\end{equation*}       
where,
\begin{equation*}
\begin{split}
\bar{I_4}:= & \sum_{k\in \mathbb{N}}\int_0^t\int_\mathscr{D}\eta^2\left[ (d+1)\sum_{i=1}^d|\delta_l^h 
       \sigma_s^{ik}|^2|\partial_i(T_l^hv_s)|^2+ (d+1)|\delta_l^h\mu_s^k \, T_l^hv_s|^2\right.\\
&  + 2  \sum_{i,j=1}^d \delta_l^h \sigma_s^{ik} \, \partial_i(T_l^hv_s) \, 
            \sigma_s^{jk} \, \partial_j(\delta_l^h v_s) +2\sum_{i,j=1}^d \delta_l^h \sigma_s^{ik} \, \partial_i(T_l^h v_s) \, \mu_s^k \, \delta_l^h v_s \\ 
& + 2\sum_{i,j=1}^d \delta_l^h \sigma_s^{ik} \, \partial_i(T_l^hv_s) \, \delta_l^h g_s^k
+ 2 \sum_{i=1}^d \sigma_s^{ik} \, \partial_i(\delta_l^h  
              v_s) \,\delta_l^h \mu_s^k \, T_l^h v_s\\ 
&  + 2 \delta_l^h \mu_s^k \, T_l^hv_s \, \mu_s^k \, \delta_l^h v_s
               + 2\delta_l^h \mu_s^k \, T_l^hv_s \, \delta_l^h g_s^k  \\     
& +|\mu_s^k \, \delta_l^h v_s|^2
       +|\delta_l^h g_s^k|^2+2\sum_{i=1}^d \sigma_s^{ik} \, \partial_i(\delta_l^h v_s) \, \mu_s^k  \,  
       \delta_l^hv_s\\
&  \left. + 2\sum_{i=1}^d \sigma_s^{ik} \, \partial_i(\delta_l^h v_s) \, \delta_l^h g_s^k
         +2 \mu_s^k \, \delta_l^h v_s \, \delta_l^h g_s^k \right]dxds
\end{split}
\end{equation*}
Substituting this in \eqref{eq:ito_difference}, we get
\begin{equation*}
\begin{split} 
\int_\mathscr{D} & \eta^2 |\delta_l^h v_t|^2 dx \\
 \leq & I_0 +I_1
       -2 \int_0^t\int_\mathscr{D}\eta ^2 \sum_{i,j=1}^d \Big[a_s^{ij}
        -\frac{1}{2}\sum_{k\in \mathbb{N}}\sigma_s^{ik}\sigma_s^{jk} \Big]  
        \partial_i(\delta_l^h v_s) \, \partial_j(\delta_l^h  v_s) dxds\\
 &  +I_2 +I_3
         + \mathscr{M}_t^h + \bar{I_4}.\\
\end{split}
\end{equation*}
which on using Assumptions A-\ref{ass:parabolicity}, A-\ref{ass:der_boundedness} (with $n=1$) and Young's inequality for an $\epsilon>0$ gives
\begin{equation}
 \begin{split} \label{eq:youngs}
  \int_\mathscr{D}  \eta^2 & |\delta_l^h v_t|^2  dx
                         \leq \int_\mathscr{D}\eta^2 |\delta_l^h v_0|^2 dx
                        -2 \kappa \int_0^t\int_\mathscr{D}\eta ^2 |\nabla(\delta_l^hv_s)|^2 dxds  
                        + \mathscr{M}_t^h\\
 + &  \int_0^t\int_\mathscr{D}\sum_{i,j=1}^d \big[ \epsilon K|\eta \partial_i(T_l^hv_s)|^2
                        + \epsilon K | \eta \partial_i(\delta_l^h v_s)|^2 
                        + \frac{C}{\epsilon} |\partial_j\eta \delta_l^h  v_s|^2\big] \,  \, dxds \\
+ & \int_0^t\int_\mathscr{D}\eta^2\left[2\delta_l^h f_s \, \delta_l^h v_s 
                        + \frac{C_{K,d}}{\epsilon}\sum_{i=1}^d|\partial_i(T_l^hv_s)|^2
                        +\frac{ C_{K,d}}{\epsilon}|T_l^hv_s|^2
                         \right. \\
&  \left. + C\sum_{k\in \mathbb{N}}|\delta_l^hg_s^k|^2+  \epsilon C_K \sum_{i=1}^d| \partial_i(\delta_l^h v_s)|^2 
                         + \frac{C_{K}}{\epsilon}|\delta_l^hv_s|^2 \right]dxds.
 \end{split}
 \end{equation} 
 Now extending $\eta, f,g$ and $v$ to $\mathbb{R}^d$ by setting them to $0$ on $\mathbb{R}^d\setminus \mathscr{D}$ and using the fact that $\supp \eta \subset \mathscr{D}$ and $\supp (T_l^{-h}\eta) \subset \mathscr{D}$ for our choice of $h$, we get
 \begin{equation}
 \begin{split}\label{eq:bound_bar_f}
 \int_\mathscr{D}\eta^2 \,  \delta_l^h f_s \,  \delta_l^h v_s dx
           &  =  \int_{\mathbb{R}^d}\eta^2 \, \delta_l^h f_s \, \delta_l^h v_s dx\\
 &  = \int_{\mathbb{R}^d}\eta^2 \,  \frac{1}{h} T_l^hf_s \, \delta_l^h v_s  
            dx-\int_{\mathbb{R}^d}\eta^2 \,  \frac{1}{h} f_s \, \delta_l^h v_s dx \\
 &   = \int_{\mathbb{R}^d} T_l^{-h}(\eta^2) \frac{1}{h}f_s \, T_l^{-h}(\delta_l^h v_s)  
            dx-\int_{\mathbb{R}^d}\eta^2 \,  \frac{1}{h} f_s \, \delta_l^h v_s  dx \\
 &   = \int_{\mathbb{R}^d}f_s\frac{1}{h}\big[T_l^{-h}(\eta^2 \delta_l^h v_s)-(\eta^2 \delta_l^h v_s)\big] dx \\
 &  = - \int_{\mathbb{R}^d}f_s \, \delta_l^{-h}(\eta^2 \, \delta_l^h v_s) dx 
= - \int_{\mathscr{D}}f_s \,\delta_l^{-h}(\eta^2 \, \delta_l^h v_s) dx \\
 &  \leq \epsilon \int_{\mathscr{D}}|\delta_l^{-h}(\eta^2 \, \delta_l^h v_s)|^2 dx  
                +\frac{1}{\epsilon}\int_{\mathscr{D}}|f_s|^2dx
\end{split}
\end{equation}
 where last inequality has been obtained using Young's inequality.

Since $\eta^2 \, \delta_l^h v_s \in H^1(\mathscr{D})$, 
using the relation between difference quotients 
and weak derivatives 
(see e.g. \cite[Ch. 5, Sec. 8, Theorem 3]{evans97}), we have
\begin{equation*}
 \int_{\mathscr{D}}|\delta_l^{-h}(\eta^2 \, \delta_l^h v_s)|^2 dx = \int_{\mathscr{D}_l^h(\eta)}|\delta_l^{-h} 
                                  (\eta^2 \, \delta_l^h v_s)|^2 dx \leq C\int_{\mathscr{D}}| 
                                  \nabla(\eta^2 \, \delta_l^h v_s)|^2 dx 
 \end{equation*}
for some constant $C$ and $\mathscr{D}_l^h(\eta):=\supp \eta \cup \supp (T_l^h\eta) \cup \supp (T_l^{-h}\eta) \Subset \mathscr{D}$.
Substituting this in \eqref{eq:bound_bar_f}, we get
\begin{equation}
\begin{split} \label{eq:bar_f}
\int_\mathscr{D}\eta^2 \, \delta_l^h f_s \, &  \delta_l^h v_s dx \leq \epsilon  C \int_{\mathscr{D}}|\nabla(\eta^2 \, \delta_l^h v_s)|^2 dx
                     +\frac{1}{\epsilon} \int_{\mathscr{D}}|f_s|^2dx \\
 & = \epsilon C \int_{\mathscr{D}}|\eta^2 \, \nabla(\delta_l^h v_s)
                      +2 \eta \, \nabla \eta \, \delta_l^h v_s|^2 dx 
                      + \frac{1}{\epsilon} \int_{\mathscr{D}}|f_s|^2 dx \\
 & \leq  \epsilon C \int_{\mathscr{D}}
                       |\eta  \,  \nabla(\delta_l^h v_s)|^2 dx 
+\epsilon  C \mu^{-2} \int_{\mathscr{D}} 
             |(\eta \, \delta_l^h v_s)|^2 dx +\frac{1}{\epsilon} \int_{\mathscr{D}}|f_s|^2dx.
\end{split}
\end{equation}
Similarly,
\begin{equation*}
\begin{split} 
  \int_\mathscr{D} &\eta^2  |T_l^h v_s|^2 dx = \int_{\mathscr{D}_l^h(\eta)} \eta^2  |T_l^h v_s|^2 dx
= \int_{\mathscr{D}_l^h(\eta)} |T_l^{-h}\eta|^2  |v_s|^2 dx \leq C  \int_{\mathscr{D}} |v_s|^2 dx
  \end{split}
\end{equation*}
and
\begin{equation*}
\begin{split} \label{eq:bound_partial_v}
\sum_{i=1}^d \int_\mathscr{D}\eta^2  |\partial_i (T_l^h v_s)|^2 dx  
& = \sum_{i=1}^d \int_{\mathscr{D}_l^h(\eta)} \eta^2  |T_l^h(\partial_i v_s)|^2 dx \\
& \leq C \sum_{i=1}^d  \int_{\mathscr{D}} |\partial_i v_s|^2 dx 
= C  \int_{\mathscr{D}} |\nabla v_s|^2 dx.
  \end{split}
\end{equation*}
Using the assumption 
$g \in  L^2(\Omega \times (0,T), \mathscr{P};H^{1}(\mathscr{D};\ell^2))$ 
and the property of difference quotients mentioned above,
\begin{equation*}
\label{eq:bound_g}
\begin{split} 
\sum_{k\in\mathbb{N}}  \int_\mathscr{D}\eta^2 |\delta_l^h \, g_s^k|^2 dx  
 =  \sum_{k\in\mathbb{N}}  \int_{\mathscr{D}_l^h(\eta)}\eta^2 |\delta_l^h g_s^k|^2 dx
\leq C\sum_{k\in\mathbb{N}}  \int_{\mathscr{D}} 
                |\nabla g_s^k|^2 dx.
\end{split}
\end{equation*}
Similarly, 
$v \in L^2(\Omega \times (0,T), \mathscr{P}; H^1(\mathscr{D}))$
and the property of difference quotients imply
 \begin{equation}
 \begin{split} \label{eq:bound_delta_v}
 & \int_\mathscr{D}  |\delta_l^h v_s|^2 dx \leq C    \int_{\mathscr{D}} |\nabla v_s|^2 dx.
  \end{split}
 \end{equation}
 Substituting \eqref{eq:bar_f}-\eqref{eq:bound_delta_v} in \eqref{eq:youngs}, we get
\begin{equation} 
  \begin{split} \label{eq: diff_quotient_est} 
 \int_\mathscr{D} \eta^2 |\delta_l^h v_t|^2 & dx 
                      \leq  C \int_\mathscr{D}|\nabla v_0|^2 dx 
                      -2 \kappa \int_0^t\int_\mathscr{D}\eta ^2 |\nabla(\delta_l^hv_s)|^2 dxds 
                      \\
 & + \mathscr{M}_t^h + \int_0^t\int_\mathscr{D} \Big[ \frac{ C_{K,d}}{\epsilon}\mu^{-2}|\nabla v_s|^2 + \epsilon C_{K}  
                      |\eta \nabla(\delta_l^hv_s)|^2 
                     +\frac{1}{\epsilon} |f_s|^2 \Big. \\
&  \Big. + \frac{C_{K,d}}{\epsilon}|v_s|^2 +C \sum_{k\in \mathbb{N}}|\nabla g_s^k|^2  
                       \Big] dxds.
 \end{split}
 \end{equation}
Further, it can be seen that the process $\mathscr{M}_t^h$
defined in \eqref{eq:ito_difference} is a local martingale where a localizing sequence of stopping times converging to $T$ as $n \to \infty$ is given by
\begin{equation} \label{eq:stop_time}
 \tau_n:=\inf \{t\in[0,T]: |\eta \delta_l^hv_s|_{L^2(\mathscr{D})}| > n\} \wedge T. 
\end{equation} 
 Thus, replacing $t$ by $t\wedge \tau_n$ in \eqref{eq: diff_quotient_est}, then taking expectation and choosing $\epsilon >0$ small enough such that $2\kappa -\epsilon C_{K}=C_\kappa >0$ and finally using Fatou's lemma, we get
\begin{equation} 
  \begin{split} \label{eq:high_bound_1}
 &\mathbb{E}\int_\mathscr{D} \eta^2 |\delta_l^h v_t|^2 dx+C_\kappa  
                    \mathbb{E}\int_0^t\int_\mathscr{D}\eta ^2 |\nabla(\delta_l^hv_s)|^2 dxds  \leq C \mathbb{E}\int_\mathscr{D}|\nabla v_0|^2 dx \\
& + \mathbb{E}\int_0^t\int_\mathscr{D} \Big[\frac{ C_{K,d}}{\epsilon}\mu^{-2}|\nabla v_s|^2  
                 + \frac{1}{\epsilon} |f_s|^2 + \frac{C_{K,d}}{\epsilon}|v_s|^2 +C \sum_{k\in  
                   \mathbb{N}}|\nabla g_s^k|^2 \Big]dxds.
 \end{split}
 \end{equation}  
Using the inequalities of Burkholder--Davis--Gundy, 
H\"older and Young together with the estimates above we get that
 \begin{equation}
 \begin{split} \label{eq:sup_M}
 &\mathbb{E}\sup_{0\leq t \leq T}|\mathscr{M}^h_{t\wedge \tau_n}|=\mathbb{E}\sup_{0\leq t \leq T}\Big|2   
                                     \sum_{k\in \mathbb{N}}\int_0^{t\wedge  
                                     {\tau_n}}\int_\mathscr{D}\eta^2 \delta_l^h(M_s^kv_s+g_s^k)   
                                      \delta_l^h v_sdxdW_s^k\Big| \\
 & \leq 4\mathbb{E} \Big(\sum_{k\in \mathbb{N}}\int_0^{\tau_n}\Big| 
                           2\int_\mathscr{D}\eta^2 \delta_l^h(M_s^kv_s+g_s^k)   
                            \delta_l^h v_sdx\Big|^2ds\Big)^\frac{1}{2} \\
 & \leq  8 \mathbb{E}\Big(\sum_{k\in \mathbb{N}}\int_0^{\tau_n} 
                           |\eta \,\delta_l^h(M_s^kv_s+g_s^k)|^2_{L^2(\mathscr{D})}   
                            |\eta \, \delta_l^h v_s|^2_{L^2(\mathscr{D})}ds\Big)^\frac{1}{2} \\
 & \leq  \frac{1}{2} \mathbb{E} \sup_{0 \leq t \leq T}|\eta \, \delta_l^h v_t| 
                           ^2_{L^2(\mathscr{D})}+C\sum_{k\in \mathbb{N}}\mathbb{E}\int_0^{\tau_n} 
                           |\eta \,\delta_l^h(M_s^kv_s+g_s^k)|^2_{L^2(\mathscr{D})} ds \\                                                      
 &  \leq  \frac{1}{2} \mathbb{E} \sup_{0 \leq t \leq T}|\eta \, \delta_l^h v_t| 
                           ^2_{L^2(\mathscr{D})}+C \zeta^{-2}\mathbb{E}\int_0^{\tau_n}\!\!\!\int_{\mathscr{D}} 
                          \! \big[|\nabla v_s|^2  
                 +|f_s|^2 +|v_s|^2 + |\nabla g_s|_{\ell^2}^2 \big]dxds.
 \end{split} 
 \end{equation}
Replacing $t$ by $t\wedge \tau_n$ in \eqref{eq: diff_quotient_est}, 
taking the supremum over $t\in[0,T]$ and using  \eqref{eq:sup_M} 
we obtain
\begin{equation*} 
  \begin{split} 
& \mathbb{E}\sup_{0\leq t \leq T}\int_\mathscr{D} \eta^2 |\delta_l^h v_{t\wedge \tau_n}|^2 dx \\
& \leq C \zeta^{-2} \left[\mathbb{E}  \int_\mathscr{D}|\nabla v_0|^2 dx  
                 + \mathbb{E}\int_0^T\int_\mathscr{D} \Big[|\nabla v_s|^2  
                   + |f_s|^2 + |v_s|^2 + |\nabla g_s|_{\ell^2}^2 \Big]\,dxds\right],
\end{split}
\end{equation*} 
which, on applying Fatou's lemma, yields
 \begin{equation*} 
  \begin{split} 
& \mathbb{E}\sup_{0\leq t \leq T}\int_\mathscr{D} \eta^2 |\delta_l^h v_t|^2 dx \\
& \leq C \zeta^{-2} \left[\mathbb{E} \int_\mathscr{D}|\nabla v_0|^2 dx  
                 + \mathbb{E}\int_0^T\int_\mathscr{D} \Big[|\nabla v_s|^2  
                   + |f_s|^2 + |v_s|^2 + |\nabla g_s|_{\ell^2}^2 \Big]\,dxds\right],
 \end{split}
 \end{equation*} 
where $C=C(K,d,\epsilon)$. Now note that the right hand side of above equation and \eqref{eq:high_bound_1} are independent of $h$ and are finite and hence using e.g.~\cite[Ch. 5, Sec. 8, Theorem 3]{evans97}), we get \eqref{eq:H_2_regularity}. 
\end{proof}

We now extend the result to the case $n=2$ as follows.
From Lemma~\ref{lem:H_2_regularity} we have that 
$v$ is a continuous $H^1(\mathscr{D'})$-valued adapted process
such that $v \in L^2(\Omega \times (0,T),\mathscr{P}; H^2(\mathscr{D'}))$,
and it satisfies~\eqref{eq:lin_spde}.
If Assumptions A-\ref{ass:der_boundedness} and  A-\ref{ass:der_initial} 
hold for $n=2$, then 
from~\eqref{eq:lin_spde}, we get
\begin{equation}
\begin{split}\label{eq:sol_high}
d(\partial_l v_t) & =\partial_l(L_tv_t+f_t)dt+\sum_{k\in \mathbb{N}}\partial_l(M_t^kv_t+g_t^k)dW_t^k\\
&= \big(L_t(\partial_l v_t)+\bar{f}_t\big)dt + \sum_{k\in \mathbb{N}}\big(M_t^k(\partial_l v_t)+\bar{g}_t^k \big)dW_t^k
 \end{split}
\end{equation}
on $[0,T] \times \mathscr{D}'$, where
\[
\bar{f}_t:=\sum_{j=1}^d \partial_j\Big(\sum_{i=1}^d \partial_la_t^{ij} \, \partial_iv_t\Big) +\sum_{i=1}^d \partial_lb_t^i \, \partial_iv_t+ \partial_lc_t \, v_t+ \partial_l f_t
\]
and
\[
\bar{g}_t^k:=\sum_{i=1}^d \partial_l\sigma_t^{ik} \, \partial_i v_t+ \partial_l\mu_t^k \, v_t+ \partial_lg_t^k.
\]
Using Assumptions A-\ref{ass:der_boundedness} , A-\ref{ass:der_initial} with $n=2$
we get that $\bar{f}\in L^2(\Omega \times (0,T),\mathscr{P};L^2(\mathscr{D}'))$ 
and $\bar{g} \in L^2(\Omega \times (0,T),\mathscr{P};H^1(\mathscr{D}';\ell^2))$.

Thus replacing $f, g^k,\mathscr{D}$ in \eqref{eq:lin_spde} 
by $\bar{f},\bar{g}^k$ and $\mathscr{D'}$ respectively, 
we see that $z=\partial_l v$ satisfies \eqref{eq:lin_spde}.
Clearly $z \in C([0,T];L^2(\mathscr{D'}))$ almost surely 
and $z \in L^2(\Omega \times (0,T); H^1(\mathscr{D'}))$ 
and hence all the assumptions of Lemma~\ref{lem:H_2_regularity}
are satisfied for the new linear equation~\eqref{eq:sol_high}. Therefore for all open $\mathscr{D''}\Subset \mathscr{D'}$ such that $\dist(\mathscr{D}'',\partial \mathscr{D}')<1$, we have
 \begin{equation*}
 \begin{split}
\mathbb{E}\sup_{0\leq t\leq T}|\partial_i z_t|_{L^2(\mathscr{D''})}^2
                                +\mathbb{E}\int_0^T & |\partial_i z_t|^2_{H^1(\mathscr{D''})} dt 
 \leq C \dist(\mathscr{D''},\partial \mathscr D')^{-2}\Bigg[\mathbb{E} \int_\mathscr{D'}|\nabla z_0|^2 dx  \\
         &        + \mathbb{E}\int_0^T\int_\mathscr{D'} \Big[|\nabla z_t|^2  
                   + |\bar{f}_t|^2 + |z_t|^2 + |\nabla \bar{g}_t|_{\ell^2}^2 \Big]dxdt\Bigg]\,.
\end{split} 
 \end{equation*}
 which, substituting back the values of $\bar{f},\bar{g}^k$ and $z=\partial_l v$ and then  using Assumption A-\ref{ass:der_boundedness} with $n=2$  and \eqref{eq:H_2_regularity}, gives

\begin{equation}\label{eq:H_3_regularity}
\begin{split}
  \mathbb{E}\sup_{0\leq t\leq T}&|\partial_i \partial_l v_t|_{L^2(\mathscr{D''})}^2
                                +\mathbb{E}\int_0^T|\partial_i \partial_l v_t|^2_{H^1(\mathscr{D''})} dt 
   \\
  \leq &  C \dist(\mathscr{D}'',\partial \mathscr D')^{-2}\Bigg[\mathbb{E} \int_\mathscr{D'}\sum_{|\gamma|\leq 2}|D^\gamma v_0|^2 dx \\ 
       &          + \mathbb{E}\int_0^T \!\!\int_\mathscr{D'} \Big[\sum_{|\gamma|\leq 2}|D^\gamma v_t|^2  
                   +\sum_{|\gamma|\leq 1}|D^\gamma f_t|^2   +\sum_{|\gamma|\leq 2} |D^\gamma g_t|_{\ell^2}^2 \Big]dxdt\Bigg]
\end{split} 
\end{equation}
for all $i=1,\ldots,d$ and open $\mathscr{D''}\Subset \mathscr{D'}$ where $C=C(d,T,K,\kappa)$.
Repeating the above procedure $k$ times, we have the following result.

\begin{lemma}\label{lem:H_k_regularity}
Assume that $v$ is a continuous $L^2(\mathscr{D})$-valued adapted process 
satisfying \eqref{eq:lin_spde} and such that $v \in L^2(\Omega \times (0,T),\mathscr{P}; H^1(\mathscr{D})).$ 
If Assumptions  A-\ref{ass:parabolicity}, A-\ref{ass:der_boundedness} 
and  A-\ref{ass:der_initial} hold for $n=k$, then
\begin{equation*}
\begin{split}
& \mathbb{E}\sup_{0\leq t\leq T}|\partial_{i_k}\ldots \partial_{i_1} v_t|_{L^2(\mathscr{D}^k)}^2
                                +\mathbb{E}\int_0^T|\partial_{i_k}\ldots \partial_{i_1} v_t|^2_{H^1(\mathscr{D}^k)} dt \leq C \dist(\mathscr{D}^k,\partial \mathscr D^{k-1})^{-2}\\
 & \Bigg[\mathbb{E} \int_{\mathscr{D}^{k-1}}\sum_{|\gamma|\leq k}|D^\gamma v_0|^2 dx  
                 + \mathbb{E}\int_0^T\int_{\mathscr{D}^{k-1}} \Big[\sum_{|\gamma|\leq k}|D^\gamma v_t|^2  
                   +\sum_{|\gamma|\leq k-1}|D^\gamma f_t|^2   \Big. \\
  & \quad  \Big. +\sum_{|\gamma|\leq k}| D^\gamma g_t|_{\ell^2}^2 \Big]dxdt\Bigg]
\end{split} 
\end{equation*}
for all $i_k=1,\ldots,d$ and open $\mathscr{D}^k\Subset \mathscr{D}^{k-1}$ such that $ \dist(\mathscr{D}^k,\partial \mathscr D^{k-1})<1$ where $C=C(d,T,K,\kappa)$.
\end{lemma}

We immediately see that Theorem~\ref{thm:H_m_regularity} follows from 
Lemma~\ref{lem:H_k_regularity}.
Using Theorems~\ref{thm:sol_bound} and~\ref{thm:H_m_regularity}, we can now prove  
Theorem~\ref{thm:higher_regularity}.

\begin{proof}[Proof of Theorem~\ref{thm:higher_regularity}]
Let $u$ be the solution to~\eqref{eq:spde} given by Theorem~\ref{thm:sol_bound}.
Then considering $f_t(u_t, \nabla u_t)+f_t^0$ as a new free term $f_t$, 
we observe that $u$ satisfies~\eqref{eq:lin_spde} with such free term. 

Now under the Assumptions A-\ref{ass:f}, A-\ref{ass:initial} and
due to Theorem~\ref{thm:sol_bound}, applied with $p \geq 2\alpha-2$, we get 
the estimate~\eqref{eq:p_bound} and hence
\begin{equation}\label{eq: free_term_bound_L^2}
\begin{split}
\mathbb{E}\int_0^T |f_t|^2  _{L^2(\mathscr{D})}dt &= \mathbb{E}\int_0^T \int_\mathscr{D}|f(u_t, \nabla u_t)+f_t^0|^2 dxdt \\
\leq & 2  \Big[\mathbb{E}\int_0^T \int_\mathscr{D} K^2(1+|u_t|)^{2\alpha-2} dxdt + \mathbb{E}\int_0^T \int_\mathscr{D}|f_t^0|^2 dxdt \Big]\\
\leq & \textcolor{red}{C} \Big[1 + \mathbb{E} \sup_{0\leq t\leq T} \int_\mathscr{D}|u_t|^{2\alpha-2} dx\Big]+ 2\mathbb{E}\int_0^T \int_\mathscr{D}|f_t^0|^2 dxdt 
<\infty.
\end{split}
\end{equation}
Hence we can apply Theorem~\ref{thm:H_m_regularity} with $n=1$ thus proving
the first claim in \textit{(i)}.
Again using \eqref{eq:H_2_regularity} for the new free term $f_t$ we get for each $i=1,\ldots,d$, 
\begin{equation*}
\begin{split}
 \mathbb{E}&\sup_{0\leq t\leq T}|\partial_i u_t|_{L^2(\mathscr{D'})}^2
+   \mathbb{E}\int_0^T|\partial_i u_t|^2_{H^1(\mathscr{D'})} dt \leq  C\dist(\mathscr{D'},\partial \mathscr{D})^{-2} \mathbb{E}\Bigg[ \int_{\mathscr D}|\nabla \phi|^2dx\\ 
&+ \int_0^T \int_{\mathscr D}\Big[|\nabla u_t|^2+|f_t|^2+|u_t|^2 + \sum_{k \in \mathbb{N}}|\nabla g_t^k|^2 \Big] dxdt\Bigg]
\end{split} 
\end{equation*}
which on using \eqref{eq: free_term_bound_L^2}, then Theorem~\ref{thm:sol_bound} with $p=2\alpha-2$ and finally H\"older's inequality proves \eqref{eq:blow1_regularity}.

Further if $f$ is a function of $t,\omega,x$ and $r$ only such that  \eqref{eq:f_prime} holds, then taking $f_t(u_t)+~f_t^0$ as a new free term $f_t$, similarly as above, we get 
\begin{equation}\label{eq: free_term_bound_H^1}
\begin{split}
&\mathbb{E}\int_0^T |\partial_i f_t|_{L^2(\mathscr{D})}^2 dt 
= \mathbb{E}\int_0^T \int_\mathscr{D}|\partial_i u_t \, \partial_r f_t(u_t)
+ \partial_i f_t(u_t)+\partial_if_t^0|^2 dxdt \\
&\leq \textcolor{red}{C} \mathbb{E}\int_0^T \int_\mathscr{D} \big[ 
|\nabla u_t|^2(1+|u_t|)^{2\alpha-4} + (1+|u_t|)^{2\alpha - 2} + |\partial_i f^0_t|^2 \big] dxdt \\
&\leq \textcolor{red}{C} \mathbb{E} \int_0^T \int_\mathscr{D}\big[1 + |\nabla u_t|^2
+|\nabla u_t|^2|u_t|^{2\alpha-4} + |u_t|^{2\alpha - 2} 
+ |\partial_if_t^0|^2 \big] dxdt <\infty
\end{split}
\end{equation}
for any $i \in \{1,\ldots,d \}$. 
Hence $f(u) + f^0$ 
is in $L^2(\Omega \times (0,T), \mathscr{P}, H^1(\mathscr{D}))$.
Thus all the conditions of Theorem~\ref{thm:H_m_regularity} are satisfied 
for $n=2$. 
This yields the first claim in \textit{(ii)}.
 Again, using \eqref{eq:H_3_regularity} for the new free term $f_t$, we obtain for each $i,j=1,\ldots,d$
\begin{equation*}
\begin{split}
  \mathbb{E}\sup_{0\leq t\leq T}&|\partial_i \partial_j u_t|_{L^2(\mathscr{D''})}^2
                                +\mathbb{E}\int_0^T|\partial_i \partial_j u_t|^2_{H^1(\mathscr{D''})} dt 
   \\
  \leq  &  C \dist(\mathscr{D}'',\partial \mathscr D')^{-2}\mathbb{E}\Bigg[ \sum_{\gamma \leq 2}\int_{\mathscr D'}|D^\gamma \phi|^2dx + \int_0^T \int_{\mathscr D'}\Big[\sum_{\gamma \leq 2}|D^\gamma u_t|^2\\ 
& + \sum_{\gamma \leq 1}|D^\gamma f_t|^2 + \sum_{\gamma \leq 2}|D^\gamma g_t|^2_{\ell^2} \Big] dxdt\Bigg] \\
\leq  &  C \dist(\mathscr{D}'',\partial \mathscr D')^{-2}\mathbb{E}\Bigg[ \sum_{\gamma \leq 2}\int_{\mathscr D'}|D^\gamma \phi|^2dx + \int_0^T \int_{\mathscr D'}\Big[\sum_{\gamma \leq 1}|D^\gamma u_t|^2\\ 
& + \sum_{\gamma \leq 1}|D^\gamma f_t|^2 + \sum_{\gamma \leq 2}|D^\gamma g_t|^2_{\ell^2} \Big] dxdt\Bigg] \\
& + C \dist(\mathscr{D}'',\partial \mathscr D')^{-2}\mathbb{E}\int_0^T \int_{\mathscr D'}\sum_{\gamma = 2}|D^\gamma u_t|^2 dxdt
\end{split} 
\end{equation*}
which on using \eqref{eq: free_term_bound_L^2}, \eqref{eq: free_term_bound_H^1}, then Theorem~\ref{thm:sol_bound} with $p=2\alpha-2$ and \eqref{eq:blow1_regularity} proves \eqref{eq:blow2_regularity}.
\end{proof}
\begin{remark} \label{rem:regularity}
Note that to prove even higher regularity than that given by Theorem~\ref{thm:higher_regularity} one would need to show that
\[
\mathbb{E}\int_0^T |\partial_j \partial_i f_t|_{L^2(\mathscr{D})}^2 dt < \infty\,.
\]
Using our approach we would require that  
\[
\mathbb{E}\int_0^T \int_{\mathscr{D}}|\partial_ju_t \partial_i u_t \partial_r^2 f_t(u_t)|^2 \, dxdt < \infty \,.
\]
However the $L^p$-estimates from Theorem~\ref{thm:sol_bound} are not sufficient.
To overcome this, one may try to formally apply $\partial_i$ to the SPDE~\eqref{eq:spde} and
then to try to get the analogous $L^p$-estimates for the equation 
for the derivative. 
However, since the semilinear term is no longer monotone,
the proof will break down.
\end{remark}

\section{Regularity in Weighted Spaces using $L^p$-theory \& Time Regularity}
\label{sec:Lp_reg} 
In this section, we raise the regularity of the solution to the SPDE \eqref{eq:spde} using $L^p$-theory from Kim~\cite{kim04}. 
The reason for using $L^p$-theory is that one gets better estimates 
for the solution of the corresponding linear equation, 
see Theorem~\ref{thm:kim}, given below, which follows immediately from Kim~\cite[Theorem 2.9]{kim04}.

We will use this together with the $L^p$-estimates 
we proved in Theorem~\ref{thm:sol_bound} to obtain regularity results 
(both space and time)
for the solution of the semilinear equation~\eqref{eq:spde},
see Theorems~\ref{thm:Lp_regularity} and~\ref{thm:time_reg} below.
In particular we obtain H\"older continuity in time 
of order $\frac{1}{2} - \frac{2}{q}$
for the solution to~\eqref{eq:spde} as a process in weighted $L^q$-space,
where $q$ comes from the integrability assumptions imposed on the data.

First, we introduce some notations, concepts and assumptions from Kim~\cite{kim04}.
For $r_0 >~ 0$ and $x\in \mathbb{R}^d$, let 
$B_{r_0}(x):=\{y \in \mathbb{R}^d : |x-y|< r_0 \}$.

\begin{definition}[Domain of class $C^1_u$]\label{def:domain}
The domain $\mathscr{D}\subset \mathbb{R}^d$ is said to be of class $C^1_u$ 
if for any $x_0 \in \partial\mathscr{D}$, there exist $r_0,K_0,L_0 >0$ and a one-one, onto continuously differentiable map $\Psi:B_{r_0}(x_0)\to G$, for a domain $G\subset \mathbb{R}^d$, satisfying the following: 
\begin{enumerate}[(i)]
\item $\Psi(x_0)=0$ and $\Psi\big(B_{r_0}(x_0)\cap \mathscr{D} \big)\subset 
\{y \in \mathbb{R}^d : y^1>0 \} $\,, 
\item $\Psi\big(B_{r_0}(x_0)\cap \partial \mathscr{D} \big)= G \cap \{ y\in \mathbb{R}^d: y^1=0$\},
\item $|\Psi|_{C^1(B_{r_0}(x_0))}\leq K_0$ and $|\Psi^{-1}(y_1)-\Psi^{-1}(y_2)|\leq K_0|y_1-y_2|$ for any $y_1,y_2 \in G$,
\item $|\Psi_x(x_1)-\Psi_x(x_2)|\leq L_0|x_1-x_2|$ for any $x_1,x_2 \in B_{r_0}(x_0)$.
\end{enumerate}  
\end{definition}
Let $\mathscr{D}$ be of class $C^1_u$ and $\rho(x):=\dist (x,\partial \mathscr{D})$. 
Then, by \cite[Lemma 2.5]{kim04} and \cite[Remark 2.7]{kim04a} (since $\mathscr{D}$ is bounded), there exists a bounded real valued function $\psi$ defined on $\bar{\mathscr{D}}$ satisfying
\begin{equation}\label{eq:psi}
\sup_{x\in \mathscr{D}}\rho^{|\gamma|}(x)|D^\gamma \partial_i\psi(x)|<\infty
\end{equation}
for any $i=1,\ldots, d$ and any multi-index $\gamma$, such that 
\[
\frac{1}{C} \rho \leq \psi \leq C \rho \,\,\,\text{in}\,\,\, \mathcal{D},
\]
for some constant $C$.
In other words, $\psi$ and $\rho$ are comparable in $\mathscr{D}$, and
in estimates they can be used interchangeably (up to multiplication by a constant).
Moreover this implies $\psi \geq 0$.

For $1 \leq q < \infty, \,\, \theta \in \mathbb{R}$ and a non-negative integer $n$, define the weighted Sobolev space $ H^{n,q}_\theta (\mathscr{D})$ by
\[
 H^{n,q}_\theta (\mathscr{D}):= \{u : \rho^{|\gamma|+(\theta-d)/q}D^\gamma u \in L^q(\mathscr{D})\,\,\, \text{for any}\,\,\, |\gamma|\leq n\}
\]
where the norm for $u \in  H^{n,q}_\theta (\mathscr{D})$ is given by
\[
|u|^q_{H^{n,q}_\theta }:= \sum_{i=0}^n \sum_{|\gamma|=i}\int_{\mathscr{D}}|D^\gamma u(x)|^q \rho^{\theta-d+iq}(x)dx . 
\] 
For functions $u:\mathbb{R}^d \to \mathbb{R}^{d'}$, we define the norm analogously and use the same notation.
The following result from Lototsky \cite{lot00} plays an important role in proving our results.
\begin{remark}\label{rem:lototsky}
The following are equivalent:
\begin{enumerate}[(i)]
\item $u \in  H^{n,q}_\theta (\mathscr{D})$ , 
\item $ u \in  H^{n-1,q}_\theta (\mathscr{D})$ and $\psi \partial_i u \in  H^{n-1,q}_\theta (\mathscr{D})$ for all $i=1,2,\ldots d$ ,
\item $ u \in  H^{n-1,q}_\theta (\mathscr{D})$ and $ \partial_i(\psi u) \in  H^{n-1,q}_\theta (\mathscr{D})$ for all $i=1,2,\ldots d$ .
\end{enumerate}
\end{remark}

Further, let
\[
\mathbb{H}^{n,q}_\theta (\mathscr{D}):=L^q(\Omega \times (0,T),\mathscr{P}, H^{n,q}_\theta (\mathscr{D}) ).
\]
In the rest of the article, we assume that
\begin{equation}\label{eq:theta}
q\geq 2 \quad \text{and} \quad d-2+q<\theta<d-1+q
\end{equation}
so that in view of~\cite[Remark 2.7]{kim04}, 
the assumption regarding existence of an $\mathscr{A}_{p,\theta}$-type set 
(see~\cite[Assumption 2.8]{kim04}), is satisfied. 
Finally, we need the following assumption on the coefficients:

\begin{assumption}\label{ass:unif_cont} 
For any $i,j=1,\ldots,d$,
\begin{enumerate}[(i)]
\item   the real valued coefficients $a^{ij}$ and their spatial derivatives up to order $n+1$ are $\mathscr{P} \times \mathscr{B}(\mathscr{D})$-measurable and bounded by $K$, 
\item the real-valued coefficients $b^i,\,\,c$ 
and their spatial derivatives up to order $n$ are 
$\mathscr{P} \times \mathscr{B}(\mathscr{D})$-measurable and are bounded by $K$, 
\item  the coefficients $\sigma^i=(\sigma^{ik})_{k=1}^\infty, \,\, \mu=(\mu^k)_{k=1}^\infty$ and their spatial derivatives up to order $n+1$ are $\ell^2$-valued $\mathscr{P} \times \mathscr{B}(\mathscr{D})$-measurable and almost surely
\[
\sum_{i=1}^d\sum_{k\in \mathbb{N}} \sum_{|\gamma|\leq n+1}|D^\gamma\sigma_t^{ik}(x)|^2+\sum_{k\in\mathbb{N}}\sum_{|\gamma|\leq n+1}|D^\gamma\mu_t^k(x)|^2 \leq K
\] 
for all $t$ and $x$, 
\item and for almost every $(t,\omega)$, the coefficients $a^{ij}(t,x)$ and $\sigma^i(t,x)$ are uniformly continuous in $x \in \mathscr{D}$.
 \end{enumerate}
\end{assumption}

Note that, the operator $L$ given by~\eqref{eq:def_L,M} is in divergence form but
the results from~\cite{kim04} are for operators in non-divergence form. 
One knows that~\eqref{eq:spde} can be expressed in non-divergence form if the
coefficients $a^{ij}$ are differentiable. 
Thus Assumption A-\ref{ass:unif_cont} implies Assumptions 2.2 and 2.3 in \cite{kim04}. 
Hence the following theorem follows from Theorem 2.9 of Kim~\cite{kim04}.

\begin{theorem}\label{thm:kim}
Assume $\mathscr{D}$ is of class $C^1_u$. 
Further, let Assumptions A-\ref{ass:parabolicity} and A-\ref{ass:unif_cont} hold
with some $n \geq 0$.
If $\psi f\in \mathbb{H}^{n,q}_\theta(\mathscr{D})$, 
$g \in \mathbb{H}^{n+1,q}_\theta(\mathscr{D};\ell^2)$ 
and $\psi^{\frac{2}{q}-1}\phi\in \mathbb{H}^{n+2,q}_\theta(\mathscr{D})$, then 
\begin{equation}
\label{eq:lin_spde_bdry}
\left\{
\begin{split}
dv_t & =(L_tv_t+f_t)dt+\sum_{k\in \mathbb{N}}(M_t^kv_t+g_t^k)dW_t^k
\,\,\,\ \text{on}\,\,\, [0,T]\times \mathscr{D},\\
v_t & =0 \,\,\, \text{on}\,\,\, \partial \mathscr{D}, \quad v_0 =\phi \,\,\, \text{on}\,\,\, \mathscr{D}
\end{split}
\right.
\end{equation}
has a unique solution $v$ such that $\psi^{-1}v\in \mathbb{H}^{n+2,q}_\theta (\mathscr{D})$.
\end{theorem}

In fact Theorem 2.9 in Kim~\cite{kim04} is proved even for fractional weighted Sobolev spaces
and under somewhat weaker assumptions. 
We do not use fractional spaces here to keep the presentation simpler. 
As to being able to use weaker assumptions: 
to obtain results for the semilinear equation~\eqref{eq:spde} we will need to apply our results from Section~\ref{sec: Lp}, 
in particular Theorem~\ref{thm:sol_bound} and thus we cannot substantially 
weaken our assumptions here.
Finally, we can state the main results on regularity for the solution to semilinear SPDE \eqref{eq:spde}.

\begin{theorem}\label{thm:Lp_regularity}
Assume $\mathscr{D}$ is of class $C^1_u$ and $u$ is the solution to \eqref{eq:spde}. Further, let Assumptions A-\ref{ass:parabolicity} to A-\ref{ass:initial} hold with $p\geq \max (q\alpha -q,2)$ and Assumption A-\ref{ass:unif_cont} holds with $n=0$. If for some $q$ satisfying \eqref{eq:theta}, $\psi^{\frac{2}{q}-1}\phi \in \mathbb{H}^{2,q}_\theta(\mathscr{D}),g~ \in~ \mathbb{H}^{1,q}_\theta(\mathscr{D};\ell^2)$ and $f^0\in \mathbb{H}_\theta^{0,q}(\mathscr{D})$, then 
$ \psi^{-1}u\in \mathbb{H}^{2,q}_\theta (\mathscr{D}).$

Moreover, in the case Assumption A-\ref{ass:unif_cont} holds with $n=1$ and almost surely
\begin{equation}\label{eq:f_kim}
\begin{split}
&|\partial_i f_t(x,r,z)| \leq K(1+|r|)^{\alpha - 1},\,\,\,
|\partial_rf_t(x,r,z)|\leq K(1+|r|)^{\alpha-2} \\
&\text{and}\,\,\,
|\partial_z f_t(x,r,z)| \leq K(1+|r|)^{\alpha - 1}
\end{split}
\end{equation}
for all $i=1,\dots, d$, $t\in[0,T], x\in \mathscr{D}, r\in \mathbb{R}$ and all $z\in \mathbb{R}^d$, if for some $q$ satisfying \eqref{eq:theta}, $\psi^{\frac{2}{q}-1}\phi \in \mathbb{H}^{3,q}_\theta(\mathscr{D})$, $g~ \in~ \mathbb{H}^{2,q}_\theta(\mathscr{D};\ell^2)$ and $f^0\in \mathbb{H}_\theta^{1,q}(\mathscr{D})$, then 
$ \psi^{-1}u\in \mathbb{H}^{3,\frac{q}{2}}_\theta (\mathscr{D})$.
\end{theorem}

\begin{remark} \label{rem:ineq_weighted}
Note that if $\psi^{-1}u\in \mathbb{H}^{2,q}_\theta (\mathscr{D})$, then by using Remark \ref{rem:lototsky}, we get
\[
\psi^{-1}u\in \mathbb{H}^{1,q}_\theta (\mathscr{D}) \,\,\, \text{and} \,\,\, \partial_i u \in \mathbb{H}^{1,q}_\theta (\mathscr{D}) \,\,\, \forall i=1,2,\ldots d.
\]
 Invoking Remark \ref{rem:lototsky} again, we have
\begin{equation}
\label{eq lototsky weights derivartives}
\psi^{-1}u\in \mathbb{H}^{0,q}_\theta (\mathscr{D}),\,\,\, \partial_i u \in \mathbb{H}^{0,q}_\theta (\mathscr{D}) \,\,\, \text{and} \,\,\, \psi \partial_i \partial_j u \in \mathbb{H}^{0,q}_\theta (\mathscr{D}) \,\,\, \forall i,j=1,2,\ldots d.
\end{equation}
\end{remark}

Finally, we present the result on time regularity of the solution of \eqref{eq:spde}.
\begin{theorem}\label{thm:time_reg}
Under the assumptions of Theorems \ref{thm:sol_bound} and \ref{thm:Lp_regularity}, 
\[
u \in C^\gamma \big([0,T]; H^{0,q}_{\theta+q} (\mathscr{D}) \big) \quad a.s.
\]
i.e., the solution $u$ to SPDE \eqref{eq:spde}, as a $H^{0,q}_{\theta+q} (\mathscr{D})$- valued process, is  H\"older continuous  of order  $\gamma$ for every $\gamma < \frac{1}{2}-\frac{2}{q}\,$  for every $q$ satisfying~\eqref{eq:theta}.
\end{theorem}

Note that one would like $u$ to be H\"older continuous with exponent $\gamma$ as a process with values in a weighted Sobolev space with the same weight exponent $\theta$ as in the results for spatial regularity (Theorem~\ref{thm:Lp_regularity}).
However we need to use~\eqref{eq lototsky weights derivartives} in our arguments
when proving Theorem~\ref{thm:time_reg} which leads to requiring the weight exponent to be $\theta + q$.

Before proving these theorems, we first prove the following lemma:

\begin{lemma}\label{lem:free_term}
Let  $\tilde{\theta} > d$ and $\tilde{q} \geq 1$. Further, let assumptions of Theorem \ref{thm:sol_bound} hold with $p~\geq~ \max (\tilde{q}\alpha - \tilde{q},2)$ and $f^0\in \mathbb{H}_{\tilde{\theta}}^{0,\tilde{q}}(\mathscr{D})$. If $u$  is the solution to~\eqref{eq:spde} and $f_t:=f_t(u_t,\nabla u_t)+~f_t^0$, then $f \in \mathbb{H}_{\tilde{\theta}}^{0,\tilde{q}}(\mathscr{D})$ and thus $\psi f \in \mathbb{H}_{\tilde{\theta}}^{0,\tilde{q}}(\mathscr{D})$.
\end{lemma} 

\begin{proof} 
First we note that $\tilde{\theta} > d$ and $\mathscr{D}$ is bounded, therefore $\sup_{x\in \mathscr{D}}\rho^{\tilde{\theta}-d}(x)<\infty$. 
Using this along with Assumption A-\ref{ass:f} implies  
\begin{equation}\label{eq:kim_reg_once}
\begin{split}
\mathbb{E}\int_0^T \int_\mathscr{D} & |f_t|^{\tilde{q}}  \rho^{\tilde{\theta}-d}dxdt =\mathbb{E}\int_0^T \int_\mathscr{D}|f_t(u_t, \nabla u_t)+f_t^0|^{\tilde{q}} \rho^{\tilde{\theta}-d}dxdt \\
& \leq  C  \Big[\mathbb{E}\int_0^T \int_\mathscr{D} (1+|u_t|)^{{\tilde{q}}\alpha-{\tilde{q}}} dxdt + \mathbb{E}\int_0^T \int_\mathscr{D}|f_t^0|^{\tilde q} \rho^{\tilde{\theta}-d}dxdt \Big]\\
& \leq  C \Big[1 + \mathbb{E} \sup_{0\leq t\leq T} |u_t|^{{\tilde{q}}\alpha-{\tilde{q}}}_{L^{{\tilde{q}}\alpha-{\tilde{q}}}}\Big]+ C\mathbb{E}\int_0^T \int_\mathscr{D}|f_t^0|^{\tilde{q}} \rho^{\tilde{\theta}-d}dxdt 
\end{split}
\end{equation}
which is finite in view of Theorem \ref{thm:sol_bound} and the fact $f^0\in \mathbb{H}_{\tilde{\theta}}^{0,{\tilde{q}}}(\mathscr{D})$. Now note that $\psi$ is bounded on $\bar{\mathscr{D}}$ and hence
\begin{equation*}
\mathbb{E}\int_0^T \int_\mathscr{D}  |\psi f_t|^q  \rho^{\theta-d}\,dxdt \leq 
C\mathbb{E}\int_0^T \int_\mathscr{D}  |f_t|^q  \rho^{\theta-d}\,dxdt<\infty \,.
\end{equation*} 
\end{proof}

\begin{proof}[Proof of Theorem~\ref{thm:Lp_regularity}]
Let $u$  be the solution to~\eqref{eq:spde} given by Theorem~\ref{thm:sol_bound}. Then 
considering $f_t(u_t,\nabla u_t)+f_t^0$ as a new free term $f_t$, the solution $u$ satisfies~\eqref{eq:lin_spde_bdry}.
We wish to apply Theorem~\ref{thm:kim} with $n=0$ and in order to do 
so we need to show that $\psi f\in \mathbb{H}^{0,q}_\theta(\mathscr{D})$.
Indeed this follows immediately by using Lemma \ref{lem:free_term} with $\tilde{\theta}=\theta$ and ${\tilde{q}}=q$. 
Hence applying Theorem~\ref{thm:kim} with $n=0$ we obtain $ \psi^{-1}u\in \mathbb{H}^{2,q}_\theta (\mathscr{D})$. 
This  completes the proof of the first statement of the theorem.

We now consider the case when Assumption A-\ref{ass:unif_cont} holds with $n=1$.
Again we will apply Theorem~\ref{thm:kim} (but now with $n=1$ and $\frac{q}{2}$ in place of $q$) and 
so we need to show that $\psi f\in \mathbb{H}^{1,\bar{q}}_\theta(\mathscr{D})$ with $\bar{q}:=\frac{q}{2}$.
Taking $\tilde{\theta}=\theta$ and ${\tilde{q}}=\bar{q}$ in Lemma \ref{lem:free_term}, we get $\psi f\in \mathbb{H}^{0,\bar{q}}_\theta(\mathscr{D})$.
Thus we consider
\begin{equation*}
\mathbb{E}\int_0^T \int_\mathscr{D}  |\partial_i \big(\psi f_t\big)|^{\bar{q}} \rho^{\theta-d+{\bar{q}}}dxdt = I_1 + I_2 \, ,
\end{equation*}
 where,
\begin{equation*}
I_1 :=\mathbb{E}\int_0^T \int_\mathscr{D} |f_t|^{\bar{q}} |\partial_i \psi|^{\bar{q}} \rho^{\theta-d+{\bar{q}}}dxdt \,\,\, \text{and} \,\,\, I_2:=\mathbb{E}\int_0^T \int_\mathscr{D}  |\partial_i f_t|^{\bar{q}} \psi^{\bar{q}} \rho^{\theta-d+{\bar{q}}}dxdt \,.
\end{equation*}
Clearly $I_1<\infty$ using  \eqref{eq:psi}, the fact $\rho$ is bounded on $\mathscr{D}$ and Lemma \ref{lem:free_term} (with $\tilde{\theta}=\theta$ and ${\tilde{q}}=\bar{q}$). 
Further observe that
\[
\begin{split}
\partial_i f_t & = \partial_i (f_t(u_t, \nabla u_t) + f^0_t) \\
& = \partial_i f_t(u_t, \nabla u_t)+\partial_i u_t \, \partial_r f_t(u_t, \nabla u_t)+\partial_i (\nabla u_t) \, \nabla_z f_t(u_t, \nabla u_t)
+\partial_i f_t^0,
\end{split}
\]
where $\nabla_z f_t$ is the gradient with respect to $z$ of $f_t=f_t(x,r,z)$.
Thus, we have 
\begin{equation} \label{eq:kim_reg_twice}
I_2 \leq C (I_3+I_4+I_5+I_6)
\end{equation}
where,
\[
I_3  := \mathbb{E}\int_0^T \int_\mathscr{D}  |\partial_i f_t(u_t, \nabla u_t)|^{\bar q} \psi^{\bar q} \rho^{\theta-d+{\bar{q}}}\, dxdt \,, 
\]
\[
I_4  := \mathbb{E}\int_0^T \int_\mathscr{D}  |\partial_i u_t \, \partial_r f_t(u_t, \nabla u_t)|^{\bar q} \psi^{\bar q} \rho^{\theta-d+{\bar{q}}}\, dxdt \,, 
\]
\[
I_5  := 
\mathbb{E}\int_0^T \int_\mathscr{D}  |\partial_i (\nabla u_t) \, \nabla_z f_t(u_t, \nabla u_t)|^{\bar q}\psi^{\bar q} \rho^{\theta-d+{\bar{q}}}\, dxdt \, ,
\]
and
\[
I_6  := \mathbb{E}\int_0^T \int_\mathscr{D}  |\partial_i f_t^0|^{\bar q} \psi^{\bar q} \rho^{\theta-d+{\bar{q}}}\, dxdt \,.
\]
Now, using the fact that $\psi$ and $\rho$ are bounded on $\mathscr{D}$ and the assumption on growth of derivatives of the semilinear term, 
see~\eqref{eq:f_kim}, we observe that 
\begin{equation*}
I_3  \leq C \mathbb{E}\int_0^T \int_\mathscr{D} \big( 1+|\partial_i f_t(u_t, \nabla u_t)| \big)^q   dxdt  \leq C \Big[ 1+ \mathbb{E}\int_0^T \int_\mathscr{D} (1+|u_t|)^{q\alpha-q} dxdt \Big]\,.
\end{equation*}
This is finite in view of Theorem \ref{thm:sol_bound}, see the estimate \eqref{eq:kim_reg_once} for details.
Further, using Young's inequality and the fact that $\psi$ and $\rho$ are bounded on $\mathscr{D}$ along with growth assumption \eqref{eq:f_kim}, we get
\begin{equation*}
\begin{split}
I_4 & \leq C \mathbb{E}\int_0^T \int_\mathscr{D} \Big[ |\partial_i u_t|^q + |\partial_r f_t(u_t, \nabla u_t)|^q \Big]  \rho^{\theta-d} dxdt \\
& \leq C \Big[|\partial_i u|_{\mathbb{H}^{0,q}_\theta}^q + \mathbb{E}\int_0^T \int_\mathscr{D} (1+|u_t|)^{q\alpha-2q} dxdt\Big]\,.
\end{split}
\end{equation*}
We see that this is finite using Remark \ref{rem:ineq_weighted} and Theorem \ref{thm:sol_bound} again.
Furthermore, using Young's inequality, growth assumption \eqref{eq:f_kim} and the fact that $\psi$ and $\rho$ are comparable, we obtain
\begin{equation*}
\begin{split}
I_5 & \leq C \mathbb{E}\int_0^T \int_\mathscr{D} \Big[ |\partial_i (\nabla u_t)|^q  + |\nabla_z f_t(u_t, \nabla u_t)|^q  \Big] \psi^q \rho^{\theta-d} dxdt \\
& \leq C \Big[|\psi \partial_i (\nabla u)|_{\mathbb{H}^{0,q}_\theta}^q + \mathbb{E}\int_0^T \int_\mathscr{D} (1+|u_t|)^{q\alpha-q} dxdt\Big].
\end{split}
\end{equation*}
Thus, applying Remark \ref{rem:ineq_weighted} and Theorem \ref{thm:sol_bound} as before, we obtain  $I_5<\infty$. 
Finally, the fact that $\psi$ and $\rho$ are comparable and bounded on $\mathscr{D}$ implies
\begin{equation*}
I_6  \leq C \mathbb{E}\int_0^T \int_\mathscr{D} \big( 1+|\partial_i f_t^0| \big)^q  \rho^{\theta-d+q}  dxdt  \leq C \Big[ 1+ \mathbb{E}\int_0^T \int_\mathscr{D} |\partial_if_t^0|^q  \rho^{\theta-d+q}dxdt \Big]
\end{equation*}
which is finite since $f^0\in \mathbb{H}_\theta^{1,q}(\mathscr{D})$. Thus $ \psi f \in \mathbb{H}_\theta^{1,\bar{q}}(\mathscr{D})$ and we can apply Theorem \ref{thm:kim} with  $n=1$ and $\bar{q}$ in place of $q$ to complete the proof.
\end{proof}


\begin{proof}[Proof of Theorem~\ref{thm:time_reg}]
We will prove the result using Kolmogorov continuity theorem.
To ease the notation we let $f_t:=f_t(u_t, \nabla u_t)+f_t^0$. 
Then from~\eqref{eq:spde} we see that
\begin{equation}\label{eq:kolmo_cont}
\mathbb{E}|u_t-u_s|_{H^{0,q}_{\theta+q}}^q \leq 2^{q-1}(I_1(s,t) + I_2(s,t)),
\end{equation}
where
\[
I_1(s,t):=\mathbb{E}\Big|\int_s^t (L_ru_r+f_r)dr\Big|^q_{H^{0,q}_{\theta+q}}
\,\,\,\text{and}\,\,\,
I_2(s,t):= \Big|\sum_{k\in\mathbb{N}}\int_s^t (M^k_ru_r+g^k_r)dW^k_r\Big|^q_{H^{0,q}_{\theta+q}}\,.
\]
We note that  $f^0\in \mathbb{H}_\theta^{0,q}(\mathscr{D})$ implies $f^0\in \mathbb{H}_{\theta+q}^{0,q}(\mathscr{D})$ because $\rho$ is bounded on $\mathscr{D}$. 
Now using H\"older's inequality, we get 
\begin{equation}\label{eq:drift_bound}
\begin{split}
I_1(s,t) & \leq (t-s)^{q-1}\mathbb{E}\int_s^t |L_ru_r+f_r|^q_{H^{0,q}_{\theta+q}}dr \\
& \leq C(t-s)^{q-1}\Big[\mathbb{E}\int_s^t |L_ru_r|^q_{H^{0,q}_{\theta+q}}dr+\mathbb{E}\int_s^t|f_r|^q_{H^{0,q}_{\theta+q}}dr\Big]\,.
\end{split} 
\end{equation}
Using Assumption A-\ref{ass:unif_cont} with $n=0$, we get 
\begin{equation*}
\begin{split}
& |L_ru_r|^q_{H^{0,q}_{\theta+q}} = \int_{\mathscr{D}} \Big|\sum_{j=1}^d\partial_j\Big(\sum_{i=1}^da_t^{ij}\partial_iu_r\Big)
+\sum_{i=1}^db_t^i\partial_iu_r+c_tu_r \Big |^q \rho^{\theta+q-d}dx \\
& \leq C \int_{\mathscr{D}}\Big( \sum_{i,j=1}^d| \partial_i\partial_ju_r|^q 
+\sum_{i=1}^d |\partial_iu_r|^q  + |u_r|^q \Big)\rho^{\theta+q-d} dx \\
& \leq C \Bigg(\sum_{i,j=1}^d|\psi \partial_i\partial_ju_r|^q_{H^{0,q}_{\theta}}+|\psi|_{C(\mathscr{\bar D})}^q \sum_{i=1}^d |\partial_iu_r|^q_{H^{0,q}_{\theta}}+|\psi|_{C(\mathscr{\bar D})}^{2q}|\psi^{-1}u_r|^q_{H^{0,q}_{\theta}}  \Bigg).
\end{split} 
\end{equation*}
Substituting this in \eqref{eq:drift_bound} and using the fact that $\psi$ is bounded on $\bar{\mathscr{D}}$, we obtain 
\begin{equation}\label{eq:drift}
\begin{split}
&I_1(s,t) \\
& \leq C(t-s)^{q-1}\Bigg(\sum_{i,j=1}^d|\psi \partial_i\partial_ju|^q_{\mathbb{H}^{0,q}_{\theta}}+ \sum_{i=1}^d |\partial_iu|^q_{\mathbb{H}^{0,q}_{\theta}}+|\psi^{-1}u|^q_{\mathbb{H}^{0,q}_{\theta}} +|f|^q_{\mathbb{H}^{0,q}_{\theta+q}} \Bigg) \\
& \leq  C(t-s)^{q-1},
\end{split} 
\end{equation}
where last statement follows using Remark \ref{rem:ineq_weighted} and Lemma \ref{lem:free_term} with $\tilde{\theta}=\theta+q$ and $\tilde{q}=q$. 

Furthermore using Burkholder--Davis--Gundy's inequality, Assumption A-\ref{ass:unif_cont} with $n=0$, H\"older's inequality and the fact that $\rho$ is bounded on $\mathscr{D}$, we see that    
\begin{equation}\label{eq:noise_bound}
\begin{split}
 I_2(s,t) & = \mathbb{E} \int_\mathscr{D}\Big|\sum_{k\in\mathbb{N}} \int_s^t  (M^k_ru_r+g^k_r)dW^k_r \Big|^q \rho^{\theta+q-d}dx\\
& \leq \int_{\mathscr{D}}\mathbb{E}\Big[\int_s^t  \sum_{k\in\mathbb{N}}|M^k_ru_r+g^k_r|^2dr\Big]^\frac{q}{2}  \rho^{\theta+q-d} dx \\
& = \int_{\mathscr{D}}\mathbb{E}\Big[\int_s^t  \sum_{k\in\mathbb{N}}\Big|\sum_{i=1}^d\sigma_r^{ik}\partial_iu_r+\mu_r^ku_r+g^k_r\Big|^2dr\Big]^\frac{q}{2}  \rho^{\theta+q-d} dx \\
& \leq C \int_{\mathscr{D}}\mathbb{E}\Big[\int_s^t  \Big(\sum_{i=1}^d|\partial_iu_r|^2+|u_r|^2+\sum_{k\in\mathbb{N}}|g^k_r|^2\Big)dr\Big]^\frac{q}{2}  \rho^{\theta+q-d} dx \\
& \leq C \int_\mathscr{D}(t-s)^{\frac{q}{2}-1}\mathbb{E} \Big[\int_s^t \Big( \sum_{i=1}^d|\partial_iu_r|^q+|u_r|^q+|g_r|_{\ell^2}^q\Big)dr\Big]  \rho^{\theta+q-d} dx \\
& \leq C(t-s)^{\frac{q}{2}-1}
\Bigg(\sum_{i=1}^d |\partial_iu|^q_{\mathbb{H}^{0,q}_{\theta}}+|\psi^{-1}u|^q_{\mathbb{H}^{0,q}_{\theta}} +|g|^q_{\mathbb{H}^{0,q}_{\theta}} \Bigg) \leq C(t-s)^{\frac{q}{2}-1}\,.
\end{split} 
\end{equation}
Here, the last inequality is obtained using Remark \ref{rem:ineq_weighted} as before and the assumption that $g\in \mathbb{H}^{1,q}_{\theta}(\mathscr{D};\ell^2)$. 
Using \eqref{eq:drift} and \eqref{eq:noise_bound} in \eqref{eq:kolmo_cont}, we obtain
\[
\mathbb{E}|u_t-u_s|_{H^{0,q}_\theta}^q \leq C |t-s|^{\frac{q}{2}-1}  
\]
which on using Kolmogorov continuity theorem concludes the result.
\end{proof}

\begin{corollary}
Under the assumptions of Theorems~\ref{thm:sol_bound},\ref{thm:higher_regularity} (parts (i) and (ii)) 
and~\ref{thm:Lp_regularity} we have
\[
u \in C^\alpha\big([0,T];H^1(\mathscr{D}')\big) \quad a.s.
\]
for every $\alpha < \frac{1}{4}-\frac{1}{q}$ with $q$ satisfying~\eqref{eq:theta} and $\mathscr{D}' \Subset \mathscr{D}$.
\end{corollary}
\begin{proof}
Note that for any open $\mathscr{D}' \Subset \mathscr{D}$, there exists a constant $M>0$ such that the distance function $\rho$ satisfies $|\rho(x)|\geq M$ for all $x \in \mathscr{D}'$. Therefore using Theorem~\ref{thm:time_reg}, we get that almost surely 
\begin{equation}\label{eq:time_interior}
\begin{split}
|u_t-u_s|_{L^q(\mathscr{D}')} &=\Big(\int_{\mathscr{D}'}|u_t-u_s|^q dx\Big)^\frac{1}{q}\leq \Big(\sup_{x\in \mathscr{D}'} \frac{1}{\rho^{\theta+q-d}}\int_{\mathscr{D}}|u_t-u_s|^q \rho^{\theta+q-d} dx\Big)^\frac{1}{q} \\
& \leq \frac{1}{\big(M^{\theta+q-d}\big)^\frac{1}{q}}|t-s|^{\frac{1}{2}-\frac{2}{q}-\epsilon}|u|_{C^{\frac{1}{2}-\frac{2}{q}-\epsilon} \big([0,T]; H^{0,q}_{\theta+q} (\mathscr{D}) \big)}
\end{split}
\end{equation}
for any $\epsilon>0$ and all $s,t\in [0,T]$.
Further, since $q\geq 2$, using H\"older's inequality we have  that there exists a random variable $C$ such that
\[
|u_t-u_s|_{L^2(\mathscr{D}')} \leq C |t-s|^{\frac{1}{2}-\frac{2}{q}-\epsilon}
\]
which implies that almost surely $u \in C^{\frac{1}{2}-\frac{2}{q}-\epsilon}\big([0,T];L^2(\mathscr{D}')\big) $ for any $\epsilon>0$. Furthermore using Theorem~\ref{thm:higher_regularity}, we have that almost surely $u \in C\big([0,T];H^2(\mathscr{D}')\big)$. Now using Gagliardo--Nirenberg inequality,  we have 
that almost surely for any $s,t \in [0,T]$
\begin{equation*}
\begin{split}
|u_t-u_s|_{H^1({\mathscr{D}'})}& \leq C |u_t-u_s|^\frac{1}{2}_{L^2(\mathscr{D}')}|u_t-u_s|^\frac{1}{2}_{H^2(\mathscr{D}')}\\
& \leq C \Big(|t-s|^{\frac{1}{2}-\frac{2}{q}-\epsilon}|u|_{C^{\frac{1}{2}-\frac{2}{q}-\epsilon} \big([0,T]; L^2(\mathscr{D}') \big)}\Big)^\frac{1}{2}\Big(2 |u|_{C\big([0,T];H^2(\mathscr{D}')\big)}\Big)^\frac{1}{2} \\
& \leq C |t-s|^{\frac{1}{4}-\frac{1}{q}-\frac{\epsilon}{2}}
\end{split}
\end{equation*}
for some random variable $C$ which concludes the result since $\epsilon>0$ is arbitrary.
\end{proof}

\section*{Acknowledgements}

The authors are grateful to the anonymous referees for their valuable suggestions which helped to significantly improve the paper.

\end{document}